\documentclass[a4paper,12pt]{article}

\usepackage{amssymb}
\usepackage{amsmath}
\usepackage{amsthm}
\newtheorem{theorem}{Theorem}
\newtheorem{lemma}{Lemma}
\newtheorem{corollary}{Corollary}
\newtheorem{prop}{Proposition}

\theoremstyle{remark}
\newtheorem{remark}{Remark}
\theoremstyle{definition}
\newtheorem{definition}{Definition}
\usepackage{color}
\usepackage{hyperref}
\hypersetup{pdfauthor={Department of Mathematics},pdfstartview={FitH},colorlinks=true,citecolor=blue,pdfborder={0 0 0}}
\usepackage{geometry}
\usepackage[cp1250]{inputenc}
\usepackage{lmodern}
\usepackage[T1]{fontenc}
\usepackage[english]{babel}    

\usepackage{graphicx}
\usepackage{placeins}
\usepackage{caption}
\usepackage{float}
\usepackage{subfigure}

\usepackage[style=english]{csquotes}

\usepackage[sort&compress]{natbib}

\usepackage{enumerate}
\usepackage{manfnt}

\newcommand*\rmd{\mathop{}\!\mathrm{d}}
\newcommand*\rme{\mathop{}\!\mathrm{e}}
\theoremstyle{plain}
\newtheorem{cond}{Condition}

\begin{document}

\renewcommand*{\thefootnote}{\fnsymbol{footnote}}

\begin{center}
\large{\textbf{Bounds on the support of the multifractal spectrum of stochastic processes}}\\
Danijel Grahovac$^1$\footnote{dgrahova@mathos.hr}, Nikolai N. Leonenko$^2$\footnote{LeonenkoN@cardiff.ac.uk}\\
\end{center}

\bigskip
\begin{flushleft}
\footnotesize{
$^1$ Department of Mathematics, University of Osijek, Trg Ljudevita Gaja 6, 31000 Osijek, Croatia\\
$^2$ School of Mathematics, Cardiff University, Senghennydd Road, Cardiff, Wales, UK, CF24 4AG}\\
\end{flushleft}

\textbf{Abstract: } Multifractal analysis of stochastic processes deals with the fine scale properties of the sample paths and seeks for some global scaling property that would enable extracting the so-called spectrum of singularities. In this paper we establish bounds on the support of the spectrum of singularities. To do this, we prove a theorem that complements the famous Kolmogorov's continuity criterion. The nature of these bounds helps us identify the quantities truly responsible for the support of the spectrum. We then make several conclusions from this. First, specifying global scaling in terms of moments is incomplete due to possible infinite moments, both of positive and negative order. For the case of ergodic self-similar processes we show that negative order moments and their divergence do not affect the spectrum. On the other hand, infinite positive order moments make the spectrum nontrivial. In particular, we show that the self-similar stationary increments process with the nontrivial spectrum must be heavy-tailed. This shows that for determining the spectrum it is crucial to capture the divergence of moments. We show that the partition function is capable of doing this and also propose a robust variant of this method for negative order moments.

\section{Introduction}
The notion of multifractality first appeared in the setting of measures. The importance of scaling relations was first stressed in the work of Mandelbrot in the context of turbulence modeling (\cite{mandelbrot1972, mandelbrot1974}). Later the notion has been extended to functions and studying fine scale properties of functions (see \cite{muzy1993multifractal, jaffard1997multifractal1, jaffard1996old}). In this setting, multifractal analysis deals with the local scaling properties of functions characterized by the Hausdorff dimension of sets of points having the same H\"older exponent. Hausdorff dimension of these sets for varying H\"older exponent yields the so-called spectrum of singularities (or multifractal spectrum). The function is called multifractal if its spectrum is nontrivial, in the sense that it is not a one point set.

However, from a practical point of view, it is impossible to numerically determine the spectrum directly from the definition. Frisch and Parisi (\cite{frisch1985fully}) were the first to propose the idea of determining the spectrum based on certain average quantities, as a numerically attainable way. In order to relate this global scaling property and the local one based on the H\"older exponents, one needs ``multifractal formalism'' to hold. This is not always the case and there has been an extensive research on this topic (see \cite{jaffard1997multifractal1, riedi1995improved, jaffard1997multifractal2, jaffard2000frisch, riedi1999multifractal}). In order to overcome the problem, one takes the other way around and seeks for different definitions of global and local scaling properties that would always be related by a certain type of multifractal formalism (see \cite{jaffard2006wavelet} for an overview in the context of measures and functions). Many authors claim that wavelets provide the best way to specify the multifractal formalism, both theoretically and numerically (see e.g. \cite{jaffard2006wavelet, bacry1993singularity}).

For stochastic processes, the local scaling properties can be immediately generalized by simply applying the definition for a function on the sample paths. As a global property, the extension is not so straightforward. In \cite{MFC1997MMAR}, the authors present a theory of multifractal stochastic processes and define the scaling property in terms of the process moments. The underlying idea is to define a scaling property more general than the well known self-similarity. However, this can lead to discrepancy. For example, $\alpha$-stable L\'evy processes with $0<\alpha<2$ are known to be self-similar with index $1/\alpha$. On the other hand, it follows from \cite{jaffard1999} that the sample paths of these processes exhibit multifractal features in the sense of the nontrivial spectrum.

The goal of this paper is to make a contribution to the multifractal theory of stochastic processes by exhibiting limitations of the existing definitions and proposing methods to overcome these. The issue of infinite moments has so far been discussed mostly as a problem of the estimation methods for determining the spectrum and has been a major critic for the partition function method. To our best knowledge, our results are the first that link heavy-tails of self-similar processes with their path irregularities in this sense. It is an intriguing fact that in this case, ignorant estimation of infinite moments will yield the correct spectrum. The bounds on the support of the spectrum we derive can be used to easily detect trivial spectrum. We do this for the class of Hermite processes. Although these bounds are very general, we later restrict our attention to stationary increments processes. We consider only $\mathbb{R}$-valued stochastic processes and our treatment is intended to be probabilistic.

The paper is organized as follows. In the next section we formally state different definitions of multifractal stochastic processes and recall some implications between them. We also discuss the multifractal formalism and different estimation methods. In Section \ref{sec3} we derive general bounds that determine the support of the multifractal spectrum and relate the bounds with the moment scaling properties. We show implications of these results for self-similar stationary increments processes. Section \ref{sec4} provides examples of stochastic processes from the perspective of different definitions. We show how the results of Section \ref{sec3} apply for each example. In Section \ref{sec5} we propose a simple modification of the partition function method that overcomes divergencies of negative order moments. We illustrate on the simulated data the advantages of this modification. Appendix contains some general facts about processes considered in Section \ref{sec4}.

\section{Definitions of the multifractal stochastic processes}\label{sec2}
In this section we provide an overview of different scaling relations that are usually referred to as multifractality. Examples of processes that satisfy these properties are given in Section \ref{sec4}. All the processes considered in this paper are assumed to be measurable, separable, nontrivial (in the sense that they are not a.s. constant) and stochastically continuous at zero, meaning that for every $\varepsilon>0$, $P(|X(h)|>\varepsilon)\to 0$ as $h\to 0$.

The best known scaling relation in the theory of stochastic processes is the self-similarity. A stochastic process $\{X(t), t \geq 0\}$ is said to be self-similar if for any $a>0$, there exists $b>0$ such that
\begin{equation*}
\{X(at)\} \overset{d}{=} \{b X(t)\},
\end{equation*}
where equality is in finite dimensional distributions. If $\{X(t)\}$ is self-similar, nontrivial and stochastically continuous at $0$, then $b$ must be of the form $a^H, a>0$, for some $H\geq 0$, i.e.
\begin{equation*}
\{X(at)\} \overset{d}{=} \{a^H X(t)\}.
\end{equation*}
A proof can be found in \cite{embrechts2002}. These weak assumptions are assumed to hold for every self-similar process considered in the paper. The exponent $H$ is usually called the Hurst parameter or index and we say $\{X(t)\}$ is $H$-ss and $H$-sssi if it also has stationary increments.\\

Following \cite{MFC1997MMAR}, the definition of a multifractal that we present first is motivated by generalizing the scaling rule of self-similar processes in the following manner:
\begin{definition}\label{defD}
A stochastic process $\{X(t)\}$ is multifractal if
\begin{equation}\label{mfdefgeneral}
\{X(ct)\} \overset{d}{=} \{M(c) X(t)\},
\end{equation}
where for every $c>0$, $M(c)$ is a random variable independent of $\{X(t)\}$ whose distribution does not depend on $t$.
\end{definition}
When $M(c)$ is non-random, then $M(c)=c^H$ and the definition reduces to $H$-self-similarity. The scaling factor $M(c)$ should satisfy the following property:
\begin{equation}\label{multiplicativeproperty}
M(ab) \overset{d}{=} M_1(a) M_2(b),
\end{equation}
for every choice of $a$ and $b$, where $M_1$ and $M_2$ are independent copies of $M$. This is sometimes called log-infinite divisibility and a motivation for this property can be found in \cite{MFC1997MMAR}. In \cite{bacry2008continuous}, the authors show that \eqref{mfdefgeneral} implies \eqref{multiplicativeproperty}.\\

However, instead of Definition \eqref{defD}, scaling is usually specified in terms of moments. The idea of extracting the scaling properties from average type quantities, like $L^p$ norm, dates back to the work of Frisch and Parisi (\cite{frisch1985fully}).
\begin{definition}\label{defM}
A stochastic process $\{X(t)\}$ is multifractal if there exist functions $c(q)$ and $\tau(q)$ such that
\begin{equation}\label{mfdefEq}
E|X(t)-X(s)|^q=c(q) |t-s|^{\tau(q)}, \quad \text{for all } t,s \in \mathfrak{T}, q \in \mathfrak{Q},
\end{equation}
where $\mathfrak{T}$ and $\mathfrak{Q}$ are intervals on the real line with positive length and $0\in \mathfrak{T}$.
\end{definition}
The function $\tau(q)$ is called the scaling function. Set $\mathfrak{Q}$ can also include negative reals. The definition can also be based on the moments of the process instead of the increments. If the increments are stationary, these definitions coincide. It is clear that if $\{X(t) \}$ is $H$-sssi then $\tau(q)=Hq$. One can also show that $\tau(q)$ must be concave. Strict concavity can hold only over a finite time horizon, otherwise $\tau(q)$ would be linear. This is not considered to be a problem for practical purposes (see \cite{MFC1997MMAR} for details). Since the scaling function is linear for self-similar processes, every departure from linearity can be attributed to multifractality. However for this reasoning to make sense, one must assume moment scaling to hold as otherwise self-similarity and multifractality are not complementary notions.

The drawback of involving moments in the definition is that they can be infinite. This narrows the applicability of the definition and as we show later, can hide the information about the singularity spectrum.

It is easy to see that under stationary increments the defining property \eqref{mfdefgeneral} along with the property \eqref{multiplicativeproperty} implies multifractality Definition \ref{defM}. Indeed, \eqref{multiplicativeproperty} implies that $E|M(c)|^q$ must be of the form $c^{\tau(q)}$ and from $X(t) \overset{d}{=} M(t) X(1)$ the claim follows. One has to assume finiteness of the moments involved in order for the statements like \eqref{mfdefEq} to have sense. Also notice that both definitions imply $X(0)=0$ a.s. which will be used through the paper.

There exist many variations of the Definition \ref{defM}. Some processes, like the classical multiplicative cascade, obey the definition only for small range of values $t$ or for asymptotically small $t$. The stationarity of increments can also be imposed. When referring to multifractality we will make clear which definition we mean. However we exclude the case of self-similar processes from the preceding definitions.\\

Definition \ref{defM} provides a simple criterion for detecting the multifractal property of the data set. Consider a stationary increments process $X(t)$ defined for $t \in [0,T]$ and suppose $X(0)=0$. Divide the interval $[0,T]$ into $\lfloor T / \Delta t \rfloor$ blocks of length $\Delta t$ and define the partition function (sometimes also called the structure function):
\begin{equation}\label{partitionfun}
S_q(T,\Delta t) = \frac{1}{\lfloor T / \Delta t \rfloor} \sum_{i=1}^{\lfloor T / \Delta t \rfloor} \left| X ( i \Delta t) -  X ( (i-1) \Delta t) \right|^q.
\end{equation}
If $\{ X(t) \}$ is multifractal with stationary increments then $E S_q(T,\Delta t)= E |X (\Delta t) |^q = c(q) {\Delta t}^{\tau(q)}$. So,
\begin{equation}\label{linearrelation}
\ln E S_q(T, \Delta t)=\tau(q) \ln {\Delta t} + \ln c(q).
\end{equation}
One can also see $S_q(T, \Delta t)$ as the empirical counterpart of the left-hand side of \eqref{mfdefEq}.

As follows from \eqref{linearrelation}, it makes sense to consider $\tau(q)$ as the slope of the linear regression of $\ln S_q(T, \Delta t)$ on $\ln {\Delta t}$. In practice, one should first check that relation \eqref{linearrelation} is valid. See \cite{FCM1997multifractalityDEM, anh2010simulation} for more details on this methodology. It was shown in \cite{GL} that a large class of processes behaves as the relation \eqref{linearrelation} holds even though there is no exact moment scaling \eqref{mfdefEq}.

Suppose that the process is sampled at equidistant time points. We can assume these are the time points $1,\dots,T$ (see \cite{GL}). By choosing points $0\leq {\Delta t}_1 < \cdots < {\Delta t}_N \leq T$ and $q_j > 0$, $j=1,\dots,M$, based on the sample $X_1,\dots,X_T$ we can calculate
\begin{equation}\label{points}
\left\{ S_{q_j} (n, \Delta t_i) \ : \ i=1,\dots, N, j=1,\dots,M \right\}.
\end{equation}
Suppose that it is checked that for fixed $q$ the points $(\ln \Delta t_i , \ln S_q(T, \Delta t))$, $i=1,\dots,n$ behave approximately linear. Using the well known formula for the slope of the linear regression line, we can define the empirical scaling function:
\begin{equation}\label{tauhat}
\hat{\tau}_{N,T}(q) = \frac{\sum_{i=1}^{N}  \ln {\Delta t_i}  \ln S_q(n,\Delta t_i) - \frac{1}{N} \sum_{i=1}^{N} \ln {\Delta t_i} \sum_{j=1}^{N} \ln S_q(n,\Delta t_i) }{ \sum_{i=1}^{N} \left(\ln {\Delta t_i}\right)^2 - \frac{1}{N} \left( \sum_{i=1}^{N} \ln {\Delta t_i} \right)^2 },
\end{equation}
where $N$ is the number of time points chosen in the regression. For reference, we state the following property as a definition.
\begin{definition}\label{defE}
A stochastic process $\{X(t)\}$ is (empirically) multifractal if it has stationary increments and the empirical scaling function \eqref{tauhat} is non-linear.
\end{definition}

\begin{remark}
Although the definition \eqref{tauhat} follows naturally from the moment scaling relation \eqref{mfdefEq}, it is not very common in the literature. Usually one tries to estimate the scaling function by using only the smallest time scale available. For example, for the cascade process on the interval $[0,T]$ the smallest interval is usually of the length $2^{-j}T$ for some $j$. One can then estimate the scaling function at point $q$ as
\begin{equation}\label{tauhatalternative}
\frac{\log_2 S_q(T,2^{-j}T)}{-j}.
\end{equation}
Estimator \eqref{tauhat} estimates the scaling function across different time scales and is therefore more general than \eqref{tauhatalternative}.
\end{remark}

\subsection{Spectrum of singularities}
Preceding definitions involve ``global'' properties of the process. Alternatively, one can base the definition on the ``local'' scaling properties, such as roughness of the process sample paths measured by the pointwise H\"older exponents. There are different approaches on how to develop the notion of a multifractal function. First, we say that a function $f: [0,\infty) \to \mathbb{R}$ is $C^{\gamma}(t_0)$ if there exists constant $C>0$ such that for all $t$ in some neighborhood of $t_0$
\begin{equation*}
|f(t)-f(t_0) | \leq C |t - t_0|^{\gamma}.
\end{equation*}
One can also define that $f$ is H\"older continuous at point $t_0$ if $|f(t)-P_{t_0} (t) | \leq C |t - t_0|^{\gamma}$ for some polynomial $P_{t_0}$ of degree at most $\lfloor \gamma \rfloor$. Two definitions coincide if $\gamma<1$. Therefore we will use the former one in this paper as in many cases we consider only functions for which $\gamma<1$ at any point. For more details see \cite{riedi1999multifractal}.

A pointwise H\"older exponent of the function $f$ at $t_0$ is then
\begin{equation}\label{pointwiseHolder}
H(t_0)= \sup \left\{ \gamma : f \in C^{\gamma}(t_0) \right\}.
\end{equation}
Consider sets $S_h=\{ t : H(t)=h \}$ where $f$ has the H\"older exponent of value $h$. These sets are usually fractal in the sense that they have non-integer Hausdorff dimension. Define $d(h)$ to be the Hausdorff dimension of $S_h$, using the convention that the dimension of an empty set is $-\infty$. Function $d(h)$ is called the spectrum of singularities (also multifractal or Hausdorff spectrum). We will refer to set of $h$ such that $d(h)\neq - \infty$ as the support of the spectrum. Function $f$ is said to be multifractal if support of its spectrum contains an interval of non-empty interior. This is naturally extended to stochastic processes:
\begin{definition}\label{defL}
A stochastic process $\{X(t)\}$ on some probability space $(\Omega, \mathcal{F}, P )$ is multifractal if for (almost) every $\omega \in \Omega$, $t \mapsto X(t,\omega)$ is a multifractal function.
\end{definition}

When considered for a stochastic process, H\"older exponents are random variables and $S_h$ random sets. However in many cases the spectrum is deterministic (\cite{balanca2013}).

\subsection{Multifractal formalism}
Multifractal formalism relates local and global scaling properties by connecting singularity spectrum with the scaling function via the Legendre transform:
\begin{equation}\label{formalism}
d(h)= \inf_q \left( hq - \tau(q) +1\right).
\end{equation}
When $d(h)=-\infty$, $h$ is not the H\"older exponent, thus the convention that $\dim_{H}(\emptyset)=-\infty$. Since the Legendre transform is concave, the spectrum is always concave function, provided multifractal formalism holds. If the multifractal formalism holds, the spectrum can be estimated as the Legendre transform of the estimated scaling function.

Substantial work has been done to investigate when this formalism holds. The validity of the formalism depends which definition of $\tau$ one uses. Since it ensures that the spectrum can be estimated from computable global quantities, it is a desirable property of the object considered. This is the reason many authors seek for different definitions of global and local scaling properties that would always be related by a certain type of multifractal formalism.

The validity of the multifractal formalism is known to be narrow when the scaling function is based on the process increments (\cite{muzy1993multifractal}). It has been showed that a large class of processes can produce nonlinear scaling function and that this behaviour is influenced by the heavy tail index (\cite{GL}). These nonlinearities are not connected with the spectrum, except in the models that posses some scaling property. In many examples negative order moments can also produce concavity since in many models they are infinite. As we will show on the example of self-similar stationary increments processes, divergence of the negative order moments has nothing to do with the spectrum. Thus the estimated nonlinearity is merely an artefact of the estimation method. We propose a simple modification of the partition function that will make it more robust. On the other hand, nonlinearity that comes from diverging positive order moments is crucial in estimating the spectrum with \eqref{formalism}. For self-similar processes, increments based partition function can capture these nonlinearities correctly.

Wavelets are considered to be the best approach to define multifractality. This is usually done by basing the definition of the partition function on the wavelet decomposition of the process (see e.g. \cite{riedi1999multifractal, audit2002wavelet}). This leads to different methods for multifractal analysis based on wavelets. However, this type of definition is also sensitive to diverging moments as has been noted in \cite{gonccalves2005diverging}, where the wavelet based estimator of the tail index is proposed. Scaling based on the wavelet coefficients is also unable to yield a full spectrum of singularities. In \cite{jaffard2004wavelet}, the formalism based on wavelet leaders has been proposed. This in some sense resembles the method we propose in Section \ref{sec5}, although our motivation comes from the results given in the next section.

On the other hand, one can also replace the definition of the spectrum to achieve multifractal formalism. For other definitions of the local scaling, such as the one based on the so-called coarse H\"older exponents, see e.g. \cite{riedi1999multifractal, CFM1997large}.

The choice of the range over which the infimum in \eqref{formalism} is taken can also be a subject of discussion. From the statistical point of view, moments of negative order are not usually investigated. Sometimes $\tau(q)$ is calculated only for $q>0$ and can therefore yield only left (increasing) part of the spectrum. For more details see \cite{riedi1999multifractal, jaffard1999}.

\section{Bounds on the support of the spectrum}\label{sec3}
The fractional Brownian motion (FBM) is a Gaussian process $\{B_H(t)\}$, which starts at zero, has zero expectation for every $t$ and the following covariance function
\begin{equation*}
E B_H(t) B_H(s) = \frac{1}{2} \left( |t|^{2H} + |s|^{2H} - |t-s|^{2H} \right), \quad H \in (0,1).
\end{equation*}
If $H=1/2$, FBM is the standard Brownian motion (BM). FBM is $H$-sssi and has a trivial spectrum consisting of only one point, i.e. $d(H)=1$, and $d(h)=-\infty$ for $h\neq H$. So there is no doubt that FBM is self-similar and not multifractal in the sense of all definitions considered. However some self-similar processes have nontrivial spectrum. Our goal in this section is to identify the property of the process that makes the spectrum nontrivial.

We do this by deriving the bounds on the support of the spectrum. The lower bound is a consequence of the well-known Kolmogorov's continuity theorem. For the upper bound we prove a sort of complement of this theorem.

Before we proceed, we fix the following notation for some general process $\{X(t), t \in [0,T]\}$. We denote the range of finite moments as $\mathfrak{Q}=(\underline{q},\overline{q})$, i.e.
\begin{equation}\label{qLU}
\begin{aligned}
\overline{q} &= \sup \{ q >0 : E|X(T)|^q < \infty \},\\
\underline{q} &= \inf \{ q <0 : E|X(T)|^q < \infty \}.
\end{aligned}
\end{equation}
If $\{X(t)\}$ is multifractal in the sense of Definition \ref{defM} with the scaling function $\tau$ define
\begin{equation}\label{H-+}
\begin{aligned}
H^- &= \sup \left\{ \frac{\tau(q)}{q} - \frac{1}{q} : q \in (0, \overline{q}) \  \& \ \tau(q)>1 \right\},\\
\widetilde{H^+} &= \inf \left\{ \frac{\tau(q)}{q} - \frac{1}{q} : q \in (\underline{q},0) \  \& \ \tau(q)<1 \right\}.
\end{aligned}
\end{equation}

\subsection{The lower bound}
Using the well known Kolmogorov's criterion it is easy to derive the lower bound on the support of the spectrum. The proof of the following theorem can be found in \cite[Theorem 2.8]{karatzas1991brownian}.

\begin{theorem}[Kolmogorov-Chentsov]\label{thm:Kolmogorov-Chentsov}
Suppose that a process $\{X(t), t \in [0,T]\}$ satisfies
\begin{equation}\label{kolmomcrit}
E|X(t)-X(s)|^{\alpha} \leq C |t-s|^{1+\beta},
\end{equation}
for some positive constants $\alpha,\beta,C$. Then there exists a modification $\{\tilde{X}(t), t \in [0,T]\}$ of $\{X(t)\}$, which is locally H\"older continuous with exponent $\gamma$ for every $\gamma \in (0,\beta/\alpha)$. This means that there exists some a.s. positive random variable $h(\omega)$ and constant $\delta>0$ such that
\begin{equation*}
P \left( \omega : \sup_{|t-s|<h(\omega), \ s,t \in [0,T]} \frac{|\tilde{X}(t,\omega)-\tilde{X}(s,\omega)|}{|t-s|^{\gamma}} \leq \delta \right)=1.
\end{equation*}
\end{theorem}

\bigskip

\begin{prop}\label{prop1}
Suppose $\{X(t), t \in [0,T]\}$ is multifractal in the sense of Definition \ref{defM}. If for some $q>0$, $E|X(T)|^q<\infty$ and $\tau(q)>1$, then there exists a modification of $\{X(t)\}$ which is locally H\"older continuous with exponent $\gamma$ for every
\begin{equation*}
\gamma \in \left(0,\frac{\tau(q)}{q} - \frac{1}{q} \right).
\end{equation*}
In particular, there exist a modification such that for almost every sample path,
\begin{equation*}
H^- \leq H(t) \quad \text{ for each } t \in [0,T],
\end{equation*}
where $H(t)$ is defined by \eqref{pointwiseHolder} and $H^-$ by \eqref{H-+}.
\end{prop}

\begin{proof}
This is a simple consequence of Theorem \ref{thm:Kolmogorov-Chentsov} since Definition \ref{defM} implies
\begin{equation*}
E|X(t)-X(s)|^q=c(q) |t-s|^{1+(\tau(q)-1)}.
\end{equation*}
Fixing $s$ in the definition of the local H\"older exponent gives the pointwise H\"older exponent.
\end{proof}

In the sequel we always suppose to work with the modification from Proposition \ref{prop1}. We can conclude that the spectrum $d(h)=-\infty$ for $h \in (0,H^-)$. This way we can establish an estimate for the left endpoint of the interval where the spectrum is defined. It also follows that if the process is $H$-sssi and has finite moments of every positive order, then $H^-=H\leq H(t)$. Thus, when moment scaling holds, path irregularities are closely related with infinite moments of positive order. We make this point stronger later.

Theorem \ref{thm:Kolmogorov-Chentsov} is valid for general stochastic processes. Although moment condition \eqref{kolmomcrit} is appealing, the condition needed for the proof of Theorem \ref{thm:Kolmogorov-Chentsov} can be stated in a different form. If we assume stationarity of the increments, other forms can also be derived. Some of them may seem strange at the moment but will prove to be useful later on.

\begin{lemma}\label{lemma:kolconditions}
Suppose that $\{X(t), t \in [0,T]\}$ is a stochastic process. Then there exists a modification of $\{X(t)\}$ which is a.s. locally H\"older continuous of order $\gamma>0$ if any of the following holds:
\newcounter{saveenum}
\begin{enumerate}[(i)]
  \item for some $\eta>1$ it holds that for every $s\in [0,T)$ and $C>0$
  \begin{equation*}
  P \left( \left| X(s+t) - X(s) \right| \geq C t^{\gamma} \right) = O(t^{\eta}), \quad \text{ as } t \to 0,
  \end{equation*}
  \item for some $m \in \mathbb{N}$, $\eta>1$  it holds that for every $s\in [0,T)$ and $C>0$
  \begin{equation*}
  P \left(\max_{l=1,\dots,m} \left| X(s+lt) - X(s+(l-1)t) \right| \geq C t^{\gamma} \right) = O(t^{\eta}), \quad \text{ as } t \to 0
  \end{equation*}
  \item for some $m \in \mathbb{N}$, $\alpha>0$ and $\beta > \alpha \gamma + 1$ it holds that for every $s\in [0,T)$
  \begin{equation*}
  E \left[  \max_{l=1,\dots,m} \left| X(s+lt) - X(s+(l-1)t) \right| \right]^{\alpha} = O \left( t^{ \beta} \right), \quad \text{ as } t \to 0.
  \end{equation*}
\end{enumerate}
If $\{X(t)\}$ has stationary increments it is enough to consider only $s=0$.
\end{lemma}

\begin{proof}
That $(i)$ is sufficient is obvious from the proof of Theorem \ref{thm:Kolmogorov-Chentsov}; see \cite[Theorem 2.8]{karatzas1991brownian}. Since $m$ is fixed it is easy to see that $(ii)$ implies $(i)$. That $(iii)$ implies $(ii)$ follows from the Chebyshev's inequality.
\end{proof}

\subsection{The upper bound}

It is considered that the negative order moments determine the right part of the spectrum. We show that this is only partially true, as this depends on whether the negative order moments are finite. To establish the bound on the right endpoint of the spectrum, one needs to show that sample paths are nowhere H\"older continuous of some order $\gamma$, i.e. that a.s. $t \mapsto X_t \notin C^{\gamma}(t_0)$ for each $t_0\in [0,T]$. To show this we first use a criterion based on the negative order moments, similar to \eqref{kolmomcrit}. The resulting theorem can be seen as a sort of a complement of the Kolmogorov-Chentsov theorem. We then apply this to moment scaling multifractals to get an estimate for the support of the spectrum.

\begin{theorem}\label{thm:ComplementKolmogorov-Chentsov}
Suppose that a process $\{X(t), t \in [0,T]\}$ defined on some probability space $(\Omega,\mathcal{F}, P)$ satisfies
\begin{equation}\label{kolmomcritnega}
E|X(t)-X(s)|^{\alpha} \leq C |t-s|^{1+\beta},
\end{equation}
for all $t,s \in [0,T]$ and for some constants $\alpha<0$, $\beta<0$ and $C>0$. Then, for $P$-a.e. $\omega\in \Omega$ it holds that for each $\gamma > \beta/\alpha$ the path $t \mapsto X_t(\omega)$ is nowhere H\"older continuous of order $\gamma$.
\end{theorem}

\begin{proof}
If suffices to prove the statement by fixing arbitrary $\gamma > \beta/\alpha$. Indeed, this would give events $\Omega_{\gamma}$, $P(\Omega_{\gamma})=0$ such that for $\omega \in \Omega \backslash \Omega_{\gamma}$, $t \mapsto X_t(\omega)$ is nowhere H\"older continuous of order $\gamma$. If $\Omega_0$ is the union of $\Omega_{\gamma}$ over all $\gamma \in (\beta/\alpha, \infty) \cap \mathbb{Q}$, then $\Omega_0 \in \mathcal{F}$, $P(\Omega_0)=0$ and $\Omega \backslash \Omega_0$ would fit the statement of the theorem.

For notational simplicity, we assume $T=1$. For $j,k\in \mathbb{N}$ define the set
\begin{equation*}
M_{jk} := \bigcup_{t\in [0,1]} \bigcap_{h \in [0,1/k]} \left\{ \omega \in \Omega : |X_{t+h}(\omega) - X_t(\omega)| \leq j h^{\gamma} \right\}.
\end{equation*}
It is clear that if $\omega \notin M_{jk}$ for every $j,k \in \mathbb{N}$, then $t \mapsto X_t(\omega)$ is nowhere H\"older continuous of order $\gamma$. As there is countably many $M_{jk}$, it is enough to fix arbitrary $j,k \in \mathbb{N}$ and show that $M_{jk} \subset A$ for some $A\in \mathcal{F}$ such that $P(A)=0$.

Suppose $n>2 k$ and $\omega \in M_{jk}$. Then there is some $t \in [0,1]$ such that
\begin{equation}\label{doktm2pom1}
|X_{t+h}(\omega)-X_t(\omega) | \leq j h^{\gamma}, \quad \text{ for all } h\in [0,1/k].
\end{equation}
Take $i \in \{1,\dots,n\}$ such that
\begin{equation}\label{doktm2pom2}
\frac{i-1}{n} \leq t < \frac{i}{n}.
\end{equation}
Then since $n>2 k$ we have
\begin{equation*}
0 \leq \frac{i}{n} - t < \frac{i+1}{n} - t \leq \frac{i+1}{n} - \frac{i-1}{n} = \frac{2}{n},
\end{equation*}
and from \eqref{doktm2pom1} it follows
\begin{equation*}
|X_{\frac{i+1}{n}}(\omega)-X_{\frac{i}{n}}(\omega) | \leq  |X_{\frac{i+1}{n}}(\omega) - X_t( \omega)| + |X_t( \omega)-X_{\frac{i}{n}}(\omega) | \leq 2 j n^{-\gamma}.
\end{equation*}
Put $A_i^{(n)}=\left\{ |X_{\frac{i+1}{n}}-X_{\frac{i}{n}} | \leq 2 j n^{-\gamma} \right\}$. Since $\omega$ was arbitrary it follows that
$$M_{jk} \subset \bigcup_{i=1}^n A_i^{(n)}.$$
Using Chebyshev's inequality for $\alpha<0$ and the assumption of the theorem we get
\begin{equation}\label{condinproof}
\begin{aligned}
P(A_i^{(n)}) &\leq \frac{E |X_{\frac{i+1}{n}}-X_{\frac{i}{n}} |^{\alpha} }{(2j)^{\alpha} n^{-\gamma \alpha}} \leq C (2j)^{-\alpha} n^{\gamma \alpha - 1 - \beta},\\
P \left(\bigcup_{i=1}^n A_i^{(n)} \right) &\leq \sum_{i=1}^n P(A_i^{(n)}) \leq C (2j)^{-\alpha} n^{-(\beta - \gamma \alpha)}.
\end{aligned}
\end{equation}
If we set
\begin{equation*}
A = \bigcap_{n > k} \bigcup_{i=1}^n A_i^{(n)},
\end{equation*}
then $A \in \mathcal{F}$ and $M_{jk} \subset A$. Since $\gamma> \beta / \alpha$, it follows that $\beta- \gamma \alpha>0$ and hence $P(A)=0$. This proves the theorem.
\end{proof}

\begin{prop}\label{prop2}
Suppose $\{X(t), t \in [0,T]\}$ is multifractal in the sense of Definition \ref{defM}. If for some $q<0$, $E|X(T)|^q<\infty$ and $\tau(q)<1$, then almost every sample path of $\{X(t)\}$ is nowhere H\"older continuous of order $\gamma$ for every
\begin{equation*}
\gamma \in \left(\frac{\tau(q)}{q} - \frac{1}{q}, \ +\infty \right).
\end{equation*}
In particular, for almost every sample path,
\begin{equation*}
H(t) \leq \widetilde{H^+} \quad \text{ for each } t \in [0,T].
\end{equation*}
\end{prop}

\begin{proof}
Definition \ref{defM} implies
\begin{equation*}
E|X(t)-X(s)|^q=c(q) |t-s|^{1+(\tau(q)-1)}.
\end{equation*}
Since $q<0$, $\tau(q)<0$ the statement follows from Theorem \ref{thm:ComplementKolmogorov-Chentsov}.
\end{proof}

This proposition shows that $d(h)=-\infty$ for $h \in (\widetilde{H^+},\infty)$. Recall that $\widetilde{H^+}$ is defined in \eqref{H-+}.

\begin{remark}\label{remark2}
Statements like the ones in the Proposition \ref{prop1} and \ref{prop2} are stronger than saying, for example, that for every $t \in [0,T]$, $H(t) \leq C$ almost surely. Indeed, an application of the Fubini's theorem would yield that for almost every path, $H(t) \leq C$ for almost every $t$. If we put $h=C + \delta$, then the Lebesgue measure of the set $S_h=\{ t : H(t)=h \}$ is zero a.s. This, however, does not imply that $d(h)=-\infty$ and hence it is impossible to say something about the spectrum of almost every sample path. On the other hand, it is clear that this type of statements are implied by Propositions \ref{prop1} and \ref{prop2}.

For the example of this weaker type of the bound, consider $\{X(t), t \in [0,T]\}$ multifractal in the sense of Definition \ref{defM}. If for some $q<0$, $E|X(t)|^q<\infty$, then for every $t \in [0,T]$
\begin{equation*}
H(t) \leq \frac{\tau(q)}{q} \text{ a.s.}
\end{equation*}
Indeed, let $\delta>0$ and suppose $C>0$. Since $q<0$, by the Chebyshev's inequality
$$P \left( \left| X(t+\varepsilon) - X(t) \right| \leq C \varepsilon^{\frac{\tau(q)}{q}+\delta} \right) \leq \frac{E \left| X(t+\varepsilon) - X(t) \right|^q }{C^q \varepsilon^{\tau(q)+\delta q}} = \frac{c(q)}{C^q \varepsilon^{\delta q}} \to 0,$$
as $\varepsilon \to 0$. We can choose a sequence $(\varepsilon_n)$ that converges to zero such that
$$P \left( \left| X(t+\varepsilon_n) - X(t) \right| \leq C \varepsilon_n^{\frac{\tau(q)}{q}+\delta} \right) \leq \frac{1}{2^n}.$$
Now, by the Borel-Cantelli lemma
$$\frac{\left| X(t+\varepsilon_n) - X(t) \right|}{\varepsilon_n^{\frac{\tau(q)}{q}+\delta}} \to \infty \ \ a.s., \text{ as } n \to \infty.$$
Thus for arbitrary $\delta>0$ it holds that for every $t$, $H(t) \leq \frac{\tau(q)}{q}+\delta$ a.s. However this result does not allow us to say anything about the spectrum.
\end{remark}

Consider for the moment the FBM. The range of finite moments is $(-1,\infty)$ and $\tau(q)=Hq$ for $q \in (-1,\infty)$, so we have $\widetilde{H^+}=H+1$. Thus, the best we can say from Proposition \ref{prop2}, is that $d(h)=-\infty$ for $h > H+1$. However we know that $d(h)=-\infty$ for $h > H$. If the bound $\widetilde{H^+}$ could be considered over all negative order moments, we would get exactly the right endpoint of the support of the spectrum.

The fact that the bound derived in Proposition \ref{prop2} is not sharp enough for some examples points that negative order moments may not be the right paradigm to explain the spectrum. We therefore provide more general conditions that do not depend on the finiteness of moments. First of them is obvious from the proof of Theorem \ref{thm:ComplementKolmogorov-Chentsov}, Equation \eqref{condinproof}.

\begin{lemma}\label{lemma2}
Suppose that $\{X(t), t \in [0,T]\}$ is a stochastic process. Then almost every sample path of $\{X(t)\}$ is nowhere H\"older continuous of order $\gamma>0$ if for every $s\in [0,T]$ and $C>0$
\begin{equation*}
P \left( \left| X(s+t) - X(s) \right| \leq C t^{\gamma} \right) = O(t^{\eta}), \quad \text{ as } t \to 0.
\end{equation*}
with some $\eta>1$. If the increments are stationary it is enough to take $s=0$.
\end{lemma}

\bigskip

\begin{theorem}\label{thm3}
Let $\{X(t), t \in [0,T]\}$ be a stochastic process defined on some probability space $(\Omega,\mathcal{F}, P)$. Suppose that for some $\gamma>0$, $\eta>1$, $m \in \mathbb{N}$ it holds that for every $s\in [0,T]$ and $C>0$
\begin{equation}\label{thm3condition}
P \left( \max_{l=1,\dots,m} \left| X(s+lt) - X(s+(l-1)t) \right| \leq C t^{\gamma} \right) = O \left( t^{\eta} \right), \quad \text{as } t \to 0.
\end{equation}
Then, for $P$-a.e. $\omega\in \Omega$ the path $t \mapsto X_t(\omega)$ is nowhere H\"older continuous of order $\gamma$. In the stationary increments case it is enough to consider $s=0$.
\end{theorem}

\begin{proof}
The first part of the proof goes exactly as in the proof of Theorem \ref{thm:ComplementKolmogorov-Chentsov}. Fix $j,k \in \mathbb{N}$ and take $n \in \mathbb{N}$ such that
\begin{equation*}
n > (m+1) k.
\end{equation*}
If $\omega \in M_{jk}$, then there is some $t \in [0,1]$ and $i \in \{1,\dots,n\}$  such that \eqref{doktm2pom1} and \eqref{doktm2pom2} hold. Choice of $n$ ensures that for $l \in \{1,\dots,m\}$
\begin{equation*}\label{}
0< \frac{i+l-1}{n}-t < \frac{i+l}{n}-t < \frac{i-l}{n}-t + \frac{l+1}{n} \leq \frac{l+1}{n} \leq \frac{1}{k}.
\end{equation*}
It follows from \eqref{doktm2pom1} that for each $l \in \{1,\dots,m\}$
\begin{equation*}
|X_{\frac{i+l}{n}}(\omega)-X_{\frac{i+l-1}{n}}(\omega) | \leq j \left( \frac{l+1}{n} \right)^{\gamma} + j \left( \frac{l}{n} \right)^{\gamma} \leq 2 j \left( \frac{m+1}{n} \right)^{\gamma}.
\end{equation*}
Denote
\begin{align*}
A_{i,l}^{(n)} &=\left\{ |X_{\frac{i+l}{n}}-X_{\frac{i+l-1}{n}} | \leq 2 j \left( \frac{m+1}{n} \right)^{\gamma} \right\},\\
A_{i}^{(n)} &= \bigcap_{l=1}^m A_{i,l}^{(n)}.
\end{align*}
It then follows that
$$M_{jk} \subset \bigcup_{i=1}^n A_i^{(n)}.$$
From the assumption we have
\begin{align*}
P(A_i^{(n)}) &= P \left( \max_{l=1,\dots,m} |X_{\frac{i+l}{n}}-X_{\frac{i+l-1}{n}} | \leq 2 j (m+1)^{\gamma} \left( \frac{1}{n} \right)^{\gamma} \right) \leq C n^{-\eta},\\
P \left(\bigcup_{i=1}^n A_i^{(n)} \right) &\leq \sum_{i=1}^n P(A_i^{(n)}) \leq C_1 n^{-(\eta - 1)}.
\end{align*}
Now setting
\begin{equation*}
A = \bigcap_{n > k} \bigcup_{i=1}^n A_i^{(n)}  \in \mathcal{F},
\end{equation*}
it follows that $P(A)=0$, since $\eta>1$.
\end{proof}

Theorem \ref{thm3} enables one to avoid using moments in deriving the bound. As an example, we consider how Theorem \ref{thm3} can be applied in the simple case when $\{X(t)\}$ is the BM. Since $\{X(t)\}$ is $1/2$-sssi we have
\begin{equation*}
P \left( \max_{l=1,\dots,m} \left| X(lt) - X((l-1)t) \right| \leq C t^{\gamma} \right) = P \left( \max_{l=1,\dots,m} \left| X(l) - X(l-1) \right| \leq C t^{\gamma-1/2} \right).
\end{equation*}
Due to independent increments, then
\begin{equation*}
P \left( \max_{l=1,\dots,m} \left| X(l) - X(l-1) \right| \leq C t^{\gamma-1/2} \right) \leq C_1 t^{m(\gamma-1/2)},
\end{equation*}
This holds for every $\gamma>1/2$ and $m \in \mathbb{N}$ and by taking $m>1/(\gamma-1/2)$ we conclude $d(h)=-\infty$ for $h>1/2$.

Before we proceed on applying these results, we state the following simple corollary that expresses the criterion \eqref{thm3condition} in terms of negative order moments, but now moments of the maximum of increments. This is a generalization of Theorem \ref{thm:ComplementKolmogorov-Chentsov} that enables bypassing infinite negative order moments under very general conditions. From this criterion we derive in the next subsection, strong statements about the $H$-sssi processes.

\begin{corollary}\label{newComplementKC}
Suppose that a process $\{X(t), t \in [0,T]\}$ defined on some probability space $(\Omega,\mathcal{F}, P)$ satisfies
\begin{equation}\label{thmmomentcond}
E \left[  \max_{l=1,\dots,m} \left| X(s + lt) - X(s + (l-1)t) \right| \right]^{\alpha} \leq C t^{1+\beta},
\end{equation}
for all $t,s \in [0,T]$ and for some $\alpha<0$, $\beta<0$, $m \in \mathbb{N}$ and $C>0$. Then, for $P$-a.e. $\omega\in \Omega$ it holds that for each $\gamma > \beta/\alpha$ the path $t \mapsto X_t(\omega)$ is nowhere H\"older continuous of order $\gamma$.
\end{corollary}

\begin{proof}
This follows directly from the Chebyshev's inequality for negative order moments and Theorem \ref{thm3}
\begin{align*}
&P \left( \max_{l=1,\dots,m} \left| X(s+lt) - X(s+(l-1)t) \right| \leq C t^{\gamma} \right)\\
&\quad \leq C^{-\alpha} t^{-\gamma \alpha} E \left[ \max_{l=1,\dots,m} \left| X(s+lt) - X(s+(l-1)t) \right| \right]^{\alpha} = O \left( t^{-\alpha \gamma+1+\beta} \right),
\end{align*}
since $1+\beta- \alpha \gamma >1$.
\end{proof}

\subsection{The case of the self-similar stationary increments processes}

In this subsection we refine our results for the case of the $H$-sssi processes by using Corollary \ref{newComplementKC}. These results can also be viewed in the light of the classical papers \cite{vervaat1985sample, takashima1989sample}. To be able to do this, we need to make sure that the moment in \eqref{thmmomentcond} can indeed be made finite by choosing $m$ large enough. We state this condition explicitly for reference.

\begin{cond}\label{C1}
Suppose $\{X(t),t\geq 0 \}$ is a stationary increments process. For every $\alpha<0$ there is $m_0 \in \mathbb{N}$ such that
$$E \left[ \max_{l=1,\dots,m_0} \left| X(l) - X(l-1) \right| \right]^{\alpha} < \infty.$$
\end{cond}

One way of assessing the Condition \ref{C1} is given in the following lemma which is weak enough to cover all the examples considered later. Recall the definition of the range of finite moments $\underline{q}$ and $\overline{q}$ given in \eqref{qLU}.

\begin{lemma}\label{lemma:condition1}
Suppose $\{X(t),t\geq 0 \}$ is stationary increments process which is ergodic in the sense that if $E |f(X_1)| < \infty$ for some measurable $f$, then
$$\frac{\sum_{l=1}^m f(X_l-X_{l-1}) }{m} \overset{a.s.}{\to} E f(X_1), \ \text{ as } m \to \infty.$$
Suppose also that $\underline{q}<0$. Then Condition \ref{C1} holds.
\end{lemma}

\begin{proof}
Let $r<0$ be such that $E|X(1)-X(0)|^{r} <\infty$. Then
\begin{align*}
\inf_{l\in \mathbb{N}} |X(l) - X(l-1)|^r &= \lim_{m\to \infty} \min_{l=1,\dots,m} |X(l) - X(l-1)|^r \\
& \leq \lim_{m\to \infty} \frac{\sum_{l=1}^m |X(l) - X(l-1)|^r }{m} = E |X(1)-X(0)|^r=:M, \ \text{ a.s.}
\end{align*}
So,
$$\inf_{l\in \mathbb{N}} |X(l) - X(l-1)|^{\alpha} = \left( \inf_{l\in \mathbb{N}} |X(l) - X(l-1)|^{r}\right)^{\frac{\alpha}{r}} = \left( \inf_{l\in \mathbb{N}} |X(l) - X(l-1)|^r \right)^{\frac{\alpha}{r}} \leq M^{\frac{\alpha}{r}}, \ \text{ a.s.}$$
and $\inf_{l\in \mathbb{N}} |X(l) - X(l-1)|^{\alpha} $ is bounded and thus has finite expectation. Given $\alpha<0$ we can choose $m_0$ such that
\begin{align*}
&\left[ \max_{l=1,\dots,m_0} \left| X(l) - X(l-1) \right| \right]^{\alpha} = \left[  \frac{1}{\max_{l=1,\dots,m_0} \left| X(l) - X(l-1) \right|} \right]^{-\alpha}\\
&=\left[ \min_{l=1,\dots,m_0}  \frac{1}{\left| X(l) - X(l-1) \right|} \right]^{-\alpha} = \min_{l=1,\dots,m_0} |X(l) - X(l-1)|^{\alpha} \leq  M^{\frac{\alpha}{r}}, \ \text{ a.s.}
\end{align*}
which implies the statement.
\end{proof}

\begin{remark}
Two examples may provide insight of how far the assumptions of Lemma \ref{lemma:condition1} are from Condition \ref{C1}. If $X(t)=tX$ for some random variable $X$, then $\max_{l=1,\dots,m} \left| X(l) - X(l-1) \right|=X$ and thus Condition \ref{C1} depends on the range of finite moments of $X$. For the second example, suppose $X(l)-X(l-1)$ is an i.i.d. sequence such that $P(|X(1)-X(0)| \leq x) \sim 1/ \ln x$ as $x \to 0$. This implies, in particular, that $E|X(1)-X(0)|^r=\infty$ for any $r<0$. Moreover,
\begin{align*}
E \left[ \max_{l=1,\dots,m} \left| X(l) - X(l-1) \right| \right]^{\alpha} &= - \int_0^{\infty} \alpha y^{\alpha-1} P(\max_{l=1,\dots,m} \left| X(l) - X(l-1) \right| \leq y ) dy\\
&= - \int_0^{\infty} \alpha y^{\alpha-1} \frac{1}{(\ln y )^m} dy = \infty,
\end{align*}
for every $\alpha<0$ and $m \in \mathbb{N}$, thus Condition \ref{C1} does not hold.
\end{remark}

We are now ready to prove a general theorem about the $H$-sssi processes.

\begin{theorem}\label{thm:boundsHsssi}
Suppose $\{X(t), t \in [0,T]\}$ is $H$-sssi stochastic process such that Condition \ref{C1} holds and $H-1/\overline{q}\geq 0$. Then, for almost every sample path,
\begin{equation*}
H-\frac{1}{\overline{q}} \leq H(t) \leq H \quad \text{ for each } t \in [0,T].
\end{equation*}
Moreover, $d(H)=1$ a.s.
\end{theorem}

\begin{proof}
By the argument at the beginning of the proof of Theorem \ref{thm:ComplementKolmogorov-Chentsov} it is enough to take arbitrary $\gamma>H$. Given $\gamma$ we take $\alpha<1/(H-\gamma)<0$ which implies $\gamma > H -1/\alpha$. Due to Condition \ref{C1}, we can choose $m_0\in \mathbb{N}$ such that $E \left[  \max_{l=1,\dots,m_0} \left| X(lt) - X((l-1)t) \right| \right]^{\alpha} < \infty$. Self-similarity then implies that
\begin{equation*}
E \left[  \max_{l=1,\dots,m_0} \left| X(lt) - X((l-1)t) \right| \right]^{\alpha} = t^{H \alpha} E \left[  \max_{l=1,\dots,m_0} \left| X(l) - X(l-1) \right| \right]^{\alpha} = C t^{1+ (H \alpha-1)}.
\end{equation*}
The claim now follows immediately from Corollary \ref{newComplementKC} with $\beta=H \alpha - 1$ since $\gamma > \beta/\alpha$.

That $H(t) \geq H-1/\overline{q}$ follows from Proposition \ref{prop1}. Since $E|X(t)|^q<\infty$ for some $q<0$ it follows from Remark \ref{remark2} that for every $t\in [0,T]$, $H(t)\leq H$ a.s. On the other hand, taking $0<q<\overline{q}$ we can get that for $\delta>0$ and $C>0$
$$P \left( \left| X(t+\varepsilon) - X(t) \right| \geq C \varepsilon^{H-\delta} \right) \leq \frac{E \left| X(t+\varepsilon) - X(t) \right|^q }{C^q \varepsilon^{Hq-\delta q}} = \frac{c(q)}{C^q \varepsilon^{-\delta q}} \to 0,$$
as $\varepsilon \to 0$. The same arguments as in Remark \ref{remark2} imply that for every $t\in [0,T]$, $H(t)\geq H$ a.s. By Fubini's theorem it follows that a.s. for almost every $t\in [0,T]$ $H(t)=H$. Thus the Lebesgue measure of the set $S_H=\{ t : H(t)=H \}$ is $1$ and so $d(H)=1$.
\end{proof}

A simple consequence of the preceding is the following statement.

\begin{corollary}\label{cor:Hsssi}
Suppose that Condition \ref{C1} holds. A $H$-sssi process with all positive order moments finite must have trivial spectrum, i.e. $d(h)=-\infty$ for $h\neq H$.
\end{corollary}

This applies to FBM, but also to all Hermite processes, like e.g. Rosenblatt process (see Section \ref{sec4}). Thus, under very general conditions a self-similar stationary increments process with a nontrivial spectrum must be heavy-tailed. This shows clearly how infinite moments can affect path properties when scaling holds. The following simple result shows this more precisely.

\begin{prop}
Suppose $\{X(t)\}$ is $H$-sssi. If $\gamma<H$ and $d(\gamma)\neq - \infty$, then $E|X(1)|^{q}=\infty$ for $q>1/(H-\gamma)$.
\end{prop}

\begin{proof}
Suppose $E|X(t)|^{q}<\infty$ for $q>1/(H-\gamma)$. Then for $\varepsilon>0$ we can apply Chebyshev's inequality to get
\begin{equation*}
P \left( |X(t)| \geq C t^{\gamma} \right) = P \left( |X(1)| \geq C t^{\gamma-H} \right) \leq \frac{E \left|X(1)\right|^{\frac{1}{H-\gamma}+\varepsilon} }{t^{-1-\varepsilon(H-\gamma)}}=O(t^{1+\varepsilon(H-\gamma)}).
\end{equation*}
By Theorem \ref{thm:Kolmogorov-Chentsov} and Lemma \ref{lemma:kolconditions} this implies $d(\gamma)=-\infty$, which is a contradiction.
\end{proof}

\subsection{The case of the multifractal processes}\label{subsec3.4}
Our next goal is to show that in the definition of $\widetilde{H^+}$ one can essentially take the infimum over all $q<0$. At the moment this makes no sense as $\tau$ from Definition \ref{defM} may not be defined in this range. It is therefore necessary to redefine the meaning of the scaling function. We therefore work with the more general Definition \ref{defD}.

In the next section we will see on the example of the log-normal cascade process that when the multifractal process has all negative order moments finite, the bound derived in Proposition \ref{prop2} is sharp. In general this would not be the case for any multifractal in the sense of Definition \ref{defD}. Take for example a multifractal random walk (MRW), which is a compound process $X(t)=B(\theta(t))$ where $B$ is BM and $\theta$ is the independent cascade process, say log-normal cascade (see \cite{bacry2003}). By the multifractality of the cascade for $t<1$, $\theta(t)=^d M(t) \theta(1)$ and multifractality of MRW implies $X(t)=^d (M(t) \theta(1))^{1/2} B(1)$. Now by independence of $B$ and $\theta$, if $E|B(1)|^q=\infty$ then $E|X(t)|^q=\infty$. Since $B(1)$ is Gaussian, moments will be infinite for $q \leq -1$.

We thus provide a more general bound which only has a restriction on the moments of the random factor from the Definition \ref{defD}. Therefore, if the process satisfies Definition \ref{defD} and if the random factor $M$ is multifractal by Definition \ref{defM} with scaling function $\tau$, we define
\begin{equation*}
H^+ = \min \left\{ \frac{\tau(q)}{q} - \frac{1}{q} : q < 0 \  \& \ E|M(t)|^q<\infty \right\}.
\end{equation*}

\begin{corollary}\label{cor:upperboundMF}
Suppose $\{X(t), t \in [0,T]\}$ is defined on some probability space $(\Omega,\mathcal{F}, P)$, has stationary increments and Condition \ref{C1} holds. Suppose also it is multifractal by Definition \ref{defD} and the random factor $M$ satisfies Definition \ref{defM} with the scaling function $\tau(q)$. If $E|M(T)|^{q} < \infty$ for $q<0$, then for $P$-a.e. $\omega\in \Omega$ it holds that for each
$$\gamma > \frac{\tau(q)}{q}-\frac{1}{q}$$
the path $t \mapsto X_t(\omega)$ is nowhere H\"older continuous of order $\gamma$.\\
In particular, for almost every sample path,
\begin{equation*}
H(t) \leq H^+ \quad \text{ for each } t \in [0,T].
\end{equation*}
\end{corollary}

\begin{proof}
By Condition \ref{C1} for $m$ large enough it follows from the multifractal property \eqref{mfdefgeneral} that
\begin{equation*}
E \left[  \max_{l=1,\dots,m} \left| X(lt) - X((l-1)t) \right| \right]^{q} = E |M(t)|^{q} E \left[  \max_{l=1,\dots,m} \left| X(l) - X(l-1) \right| \right]^{q} = C t^{1 + \tau(q)-1}.
\end{equation*}
The claim now follows from Corollary \ref{newComplementKC} with $\alpha=q$ and $\beta=\tau(q)-1$ and by the argument at the beginning of the proof of Theorem \ref{thm:ComplementKolmogorov-Chentsov}.
\end{proof}

In summary, we provide bounds on the support of the multifractal spectrum. We show that the lower bound can be derived using positive order moments and link the non-existing moments with the path properties for the case of $H$-sssi process. In general, negative order moments are not appropriate for explaining the right part of the spectrum. To derive an upper bound on the support of the spectrum, we use negative order moments of the maximum of increments. This avoids the non existence of the negative order moments which is a property of the distribution itself.

\section{Examples}\label{sec4}
In this section we list several examples of stochastic processes and investigate if the Definitions \ref{defD}-\ref{defL} hold. We show how the results of Section \ref{sec3} apply in these cases and also discuss how the multifractal formalism could be achieved. Definitions and further details on the processes considered are given in the Appendix.

\subsection{Self-similar processes}
It follows from Theorem \ref{thm:boundsHsssi} and Corollary \ref{cor:Hsssi} that if $H$-sssi process is ergodic with finite positive order moments, then the spectrum is simply
\begin{equation*}
  d(h) =
  \begin{cases}
  1, & \text{if } \ h=H\\
  -\infty, & \text{otherwise}.
  \end{cases}
\end{equation*}
This applies to all Hermite processes, e.g. BM, FBM and Rosenblatt process. Indeed, Hermite processes have all positive order moments finite and the increments are ergodic (see e.g. \cite[Section 7]{samorodnitsky2007long}). We now discuss heavy tailed examples of $H$-sssi processes.


\subsubsection{Stable processes}
Suppose $\{ X(t)\}$ is an $\alpha$-stable L\'evy motion. $\{ X(t)\}$ is $1/\alpha$-sssi and moment scaling \eqref{mfdefEq} holds but makes sense only for a range of finite moments, that is for $\mathfrak{Q}=(-1,\alpha)$ in  Definition \ref{defM}. For this range of $q$, $\tau(q)=q/\alpha$ and the process is self-similar. Due to infinite moments beyond order $\alpha$ the empirical scaling function \eqref{tauhat} will asymptotically behave for $q>0$ as
\begin{equation*}
  \tau_{\infty}(q)=
  \begin{cases}
  \frac{q}{\alpha}, & \text{if } 0<q\leq \alpha,\\
  1, & \text{if } q>\alpha.
  \end{cases}
\end{equation*}
See \cite{GL} for the precise result. Non-linearity points to multifractality in the sense of Definition \ref{defE}. The spectrum of singularities is given by (\cite{jaffard1999}):
\begin{equation*}
  d(h) =
  \begin{cases}
  \alpha h, & \text{if } \ h \in [0, 1/\alpha],\\
  -\infty, & \text{if } \ h \in (1/\alpha,+\infty].
  \end{cases}
\end{equation*}
Hence the spectrum is nontrivial and supported on $[0, 1/\alpha]$. These are exactly the bounds given in Theorem \ref{thm:boundsHsssi} as in this case $H=1/\alpha$. We stress that even self-similar processes can have multifractal paths and that this is closely related with infinite moments.

We now discuss which form of the scaling function would yield the multifractal spectrum via the Legendre transform. This will highly depend on the range of $q$ over which the infimum in the Legendre transform is taken. For example if we take infimum over $q \in [0,\alpha]$, then we get the correct spectrum from Definitions \ref{defM} and \ref{defE}. Since in practice $\alpha$ is unknown, one can take infimum over $q \in [0,+\infty)$. In this case Definition \ref{defE} yields the formalism, i.e.
\begin{equation*}
d(h)= \inf_{q \in [0,\infty)} \left( hq - \tau_{\infty}(q) +1\right).
\end{equation*}
So even though the moments beyond order $\alpha$ are infinite, estimating infinite moments with the partition function can lead to the correct spectrum of singularities.

\subsubsection{Linear fractional stable motion}
In the same manner we treat linear fractional stable motion (LFSM) (see Appendix for the definition). Dependence introduces a new parameter in the scaling relations and the spectrum. LFSM $\{ X(t)\}$ is $H$-sssi, thus is not multifractal in the sense of Definition \ref{defD}. For the range of finite moments $\mathfrak{Q}=(-1,\alpha)$, Definition \ref{defM} holds with $\tau(q)=Hq$. In this sense process is self-similar. As follows from the results of \cite{GLT} (see also \cite{HeydeSly2008}), empirical scaling function asymptotically behaves for $q>-1$ as
\begin{equation}\label{LFSMtau}
  \tau_{\infty}(q)=
  \begin{cases}
  Hq, & \text{if } 0<q\leq \alpha,\\
  1+q(H-\frac{1}{\alpha}), & \text{if } q>\alpha.
  \end{cases}
\end{equation}
The combined influence of infinite moments and dependence produces concavity, pointing to multifractality in the sense of Definition \ref{defE}. In \cite{balanca2013}, the spectrum was established for $\alpha \in [1,2)$, $H\in (0,1)$ and the long-range dependence case $H>1/\alpha$:
\begin{equation}\label{LFSMspec}
  d(h) =
  \begin{cases}
  \alpha (h-H)+1, & \text{if } \ h \in [H-\frac{1}{\alpha}, H],\\
  -\infty, & \text{otherwise}.
  \end{cases}
\end{equation}
It is known that in the case $H<1/\alpha$ sample paths are nowhere bounded which explains the assumptions. Also, increments of the LFSM are ergodic (see e.g. \cite{cambanis1987ergodic}). Since $\alpha=\overline{q}$ is the tail index, Theorem \ref{thm:boundsHsssi} gives sharp bounds on the support of the spectrum.

One can easily check that multifractal formalism can not be achieved with any of the definitions considered, except the empirical one. Indeed, it holds that
\begin{equation*}
d(h)= \inf_{q \in [0,\infty)} \left( hq - \tau_{\infty}(q) +1\right).
\end{equation*}
It is a curiosity that if we ignorantly estimate the scaling function using non-existing moments we get the correct spectrum.

\subsubsection{Inverse stable subordinator}
Inverse stable subordinator $\{X(t)\}$ is a non-decreasing $\alpha$-ss stochastic process, for some $\alpha \in (0,1)$. However application of the results of the previous section is not straightforward as it has non-stationary increments. Yet we can prove that it has a trivial spectrum defined only in point $\alpha$.

To derive the lower bound we use Theorem \ref{thm:Kolmogorov-Chentsov}. First recall that $a^{\alpha}+ b^{\alpha} \leq (a+b)^{\alpha}$ for $a,b\geq0$ and $\alpha \in (0,1)$. Taking $a=t-s$, $b=s$ when $t\geq s$ and $a=t$, $b=s-t$ when $t<s$ gives that $|t^{\alpha} - s^{\alpha}| \leq |t-s|^{\alpha}$. Since $\{X(t)\}$ has finite moments of every positive order we have for arbitrary $q>0$ and $t,s>0$
\begin{equation*}
E|X(t)-X(s)|^q = |t^{\alpha} - s^{\alpha}|^q E |X(1)|^q \leq E|X(1)|^q |t-s|^{1 + \alpha q - 1}.
\end{equation*}
By Theorem \ref{thm:Kolmogorov-Chentsov} there exists modification which is a.s. locally H\"older continuous of order $\gamma< \alpha-1/q$. Since $q$ can be taken arbitrarily large, we can get the modification such that a.s. $H(t) \geq \alpha$ for every $t \in [0,T]$.

For the lower bound we use Theorem \ref{thm3}. Given $\gamma>\alpha$ we choose $m \in \mathbb{N}$ such that $m>1/(\gamma - \alpha)$. If $\{Y(t)\}$ is the corresponding stable subordinator, from the property $\{ X(t) \leq a \} = \{ Y(a) \geq t \}$ we have for every $t_1<t_2$ and $a>0$
$$\{X(t_2)-X(t_1) \leq a\}= \{ Y_{X(t_1)+a} \geq t_2 \} = \{ Y_{X(t_1)+a} -t_1 \geq t_2 - t_1\}.$$
By \cite[Theorem 4, p. 77]{bertoin1998levy}, for every $t_1>0$, $P(Y_{X(t_1)}>t_1)=1$, thus on this event
$$\{ Y_{X(t_1)+a} -t_1 \geq t_2 - t_1\} \subset \{ Y_{X(t_1)+a} - Y_{X(t_1)} \geq t_2 - t_1\}.$$
Now by the strong Markov property choosing $t$ small enough and stationarity of increments of $\{Y(t)\}$ we have
\begin{align*}
&P \left( \max_{l=1,\dots,m} \left| X(s+lt) - X(s+(l-1)t) \right| \leq C t^{\gamma} \right) \\
&= P \left( X(s+t) - X(s)  \leq C t^{\gamma}, \dots, X(s+mt) - X(s+(m-1)t) \leq C t^{\gamma}\right)\\
&\leq P \left( Y_{X(s)+C t^{\gamma}} - Y_{X(s)} \geq t, \dots, Y_{X(s+(m-1)t)+Ct^{\gamma}} - Y_{X(s+(m-1)t)} \geq t \right)\\
&\leq \left( P \left( Y(Ct^{\gamma}) \geq t \right) \right)^m = \left( P \left( Y(1) \geq C^{-\frac{1}{\alpha}} t^{1- \frac{\gamma}{\alpha}} \right) \right)^m \leq \left( C_1 t^{\gamma-\alpha} \right)^m,
\end{align*}
by the regular variation of the tail. Due to choice of $m$, $m(\gamma-\alpha)>1$. This property of the first-passage process has been noted in \cite[p. 96]{bertoin1998levy}

\subsection{L\'evy processes}
Suppose $\{X(t), t \geq 0\}$ is a L\'evy process. The L\'evy processes in general do not satisfy the moment scaling of the form \eqref{mfdefEq}. The only such examples are the BM and the $\alpha$-stable L\'evy process. It was shown in \cite{GL} that the data from these processes will behave as obeying the scaling relation \eqref{linearrelation}. If $X(1)$ is zero mean with heavy-tailed distribution with tail index $\alpha$ and if $\Delta t_i$ in \eqref{tauhat} is of the form $T^{\frac{i}{N}}$ for $i=1,\dots,N$, then for every $q>0$ as $T,N \to \infty$ the empirical scaling function will asymptotically behave as
\begin{equation}\label{tau}
\tau_{\infty}(q)=
\begin{cases}
\frac{q}{\alpha}, & \text{if } 0<q\leq \alpha \ \& \ \alpha\leq 2,\\
1, & \text{if } q>\alpha \ \& \ \alpha\leq 2,\\
\frac{q}{2}, & \text{if } 0<q\leq \alpha \ \& \ \alpha> 2,\\
\frac{q}{2}+\frac{2(\alpha-q)^2 (2\alpha+4q-3\alpha q)}{\alpha^3 (2-q)^2}, & \text{if } q>\alpha \ \& \ \alpha> 2.
\end{cases}
\end{equation}
See \cite{GL} and \cite{GJLT} for the proof and more details. This shows that estimating the scaling function under infinite moments is influenced by the value of the tail index $\alpha$ and will yield a concave shape of the scaling function.

Local regularity of L\'evy processes has been established in \cite{jaffard1999} and extended in \cite{balanca2013} under weaker assumptions. Denote by $\beta$ the Blumenthal-Getoor (BG) index of a L\'evy process, i.e.
$$\beta=\inf \left\{ \gamma\geq 0 : \int_{|x|\leq 1} |x|^{\gamma} \pi (dx) < \infty \right\},$$
where $\pi$ is the corresponding L\'evy measure. If $\sigma$ is a Brownian component of the characteristic triplet, define
\begin{equation*}
\beta' =
\begin{cases}
\beta , & \text{if } \ \sigma=0,\\
2, & \text{if } \ \sigma\neq 0.\\
\end{cases}
\end{equation*}
Multifractal spectrum of the L\'evy process is given by
\begin{equation}\label{spectrumLevyP}
d(h) =
\begin{cases}
\beta h, & \text{if } \ h \in [0, 1/\beta'),\\
1, & \text{if } \ h = 1/\beta',\\
-\infty, & \text{if } \ h \in (1/\beta',+\infty].
\end{cases}
\end{equation}
Thus the most L\'evy processes have a non-trivial spectrum. Moreover, the estimated scaling function and the spectrum are not related as they depend on the different parts of the L\'evy measure. Behaviour of the estimated scaling function is governed by the tail index which depends on the behaviour of the L\'evy measure at infinity since for $q>0$, $E|X(1)|^q<\infty$ is equivalent to $\int_{|x|>1} |x|^q \pi (dx) < \infty$. On the other hand, the spectrum is determined by the behaviour of $\pi$ around origin, i.e. by the BG index. The discrepancy happens as there is no exact scaling in the sense of \eqref{mfdefgeneral} or \eqref{mfdefEq}. If there is an exact scaling property, like in the case of the stable process, spectrum can be estimated correctly. It is therefore important to check the validity of relation \eqref{linearrelation} from the data. This may be problematic as it is hard to distinguish exact scaling from the asymptotic one exhibited by a large class of processes.

As there is no exact moment scaling, Propositions \ref{prop1} and \ref{prop2} generally do not hold. Thus, in order to establish bounds on the support of the spectrum we use other criteria from Section \ref{sec3}. We present two analytically tractable examples to illustrate the use of these criteria.

\subsubsection{Inverse Gaussian L\'evy process}
Inverse Gaussian L\'evy process is a subordinator such that $X(1)$ has an inverse Gaussian distribution $IG(\delta, \lambda)$, $\delta>0, \lambda\geq 0$, given by the density
\begin{equation*}
f(x)=\frac{\delta}{\sqrt{2\pi}} \rme^{\delta \lambda} x^{-3/2} \exp \left\{ -\frac{1}{2} \left( \frac{\delta^2}{x} + \lambda^2 x \right) \right\}, \quad x>0.
\end{equation*}
The expression for the cumulant reveals that for each $t$ $X(t)$ has $IG(t\delta,\lambda)$ distribution. L\'evy measure is absolutely continuous with density given by
\begin{equation*}
g(x)=\frac{\delta}{\sqrt{2\pi}} y^{-3/2} \exp \left\{ -\frac{\lambda^2 y }{2} \right\}, \quad x>0,
\end{equation*}
thus the BG index is $\beta=1/2$. See \cite{eberlein2004generalized} for more details. Inverse Gaussian distribution has moments of every order finite and for every $q \in \mathbb{R}$ we can express them as
\begin{align*}
E |X(1)|^q &= \int_0^{\infty} x^q f(x) dx = \frac{\delta}{\sqrt{2\pi}} \rme^{\delta \lambda} \left( \frac{2}{\lambda^2} \right)^{q-1/2}  \int_0^{\infty} x^{q-3/2} \exp \left\{ -x - \frac{\delta^2 \lambda^2}{4 x} \right\} dx\\
&=\frac{\delta}{\sqrt{2\pi}} \rme^{\delta \lambda} \left( \frac{2}{\lambda^2} \right)^{q-1/2} K_{-q+\frac{1}{2}} ( \delta \lambda ) 2 \left( \frac{\delta \lambda}{2} \right)^{q-\frac{1}{2}} = \sqrt{\frac{2}{\pi}} \rme^{\delta \lambda} \delta^{q+\frac{1}{2}} \lambda^{-q+\frac{1}{2}} K_{-q+\frac{1}{2}} ( \delta \lambda ).
\end{align*}
where we have used \cite[Equation 10.32.10]{olver2010} and $K_{\nu}$ denotes the modified Bessel function of the second kind. This implies that
\begin{equation*}
E |X(t)|^q = \sqrt{\frac{2}{\pi}} \rme^{t \delta \lambda} t^{q+\frac{1}{2}}  \delta^{q+\frac{1}{2}} \lambda^{-q+\frac{1}{2}} K_{-q+\frac{1}{2}} ( t \delta \lambda ) \sim C t^{q+\frac{1}{2}} t^{-|-q+\frac{1}{2}|}, \ \text{ as } t \to 0,
\end{equation*}
where we have used $K_{\nu}(z) \sim \frac{1}{2} \Gamma(\nu) (\frac{1}{2} z )^{-\nu}$ for $z>0$ and $K_{-\nu}(z)=K_{\nu}(z)$. For any choice of $\gamma>0$ condition (i) of Lemma \ref{lemma:kolconditions} cannot be fulfilled, so the best we can say is that the lower bound is $0$, in accordance with \eqref{spectrumLevyP}. Since negative order moments are finite, Lemma \ref{lemma2} yields the sharp upper bound on the spectrum. Indeed, given $\gamma>1/\beta=2$ we have for $q<1/(2-\gamma)<0$
\begin{equation*}
P \left( |X(T) \leq C t^{\gamma} \right) \leq \frac{E|X(t)|^q}{t^{\gamma q}} \leq C_1 t^{-q (\gamma-2)}.
\end{equation*}
It follows that the upper bound is $2$ which is exactly the reciprocal of the BG index.

\subsubsection{Tempered stable subordinator}
Positive tempered stable distribution is obtained by exponentially tilting the L\'evy density of the totally skewed $\alpha$-stable subordinator, $0<\alpha<1$. Tempered stable subordinator is a L\'evy process $\{X(t)\}$ such that $X(1)$ has positive tempered stable distribution given by the cumulant function
\begin{equation*}
\Phi(\theta) = \log E \left[ \rme^{-\theta X(1)} \right] = \delta \lambda - \delta \left( \lambda^{1/\alpha} + 2 \theta \right)^{\alpha},\quad \theta \geq 0,
\end{equation*}
where $\delta$ is the scale parameter of the stable distribution and $\lambda$ is the tilt parameter. In this case BG index is equal to $\alpha$ (see \cite{schoutens2003levy} for more details). We use Lemma \ref{lemma2} for $\gamma>\alpha$ to get
\begin{equation*}
P \left( |X(T) \leq C t^{\gamma} \right) \leq \rme^{-1} E \left[ \rme^{-\frac{X(t)}{Ct^{\gamma}}} \right] = \rme^{-1} \rme^{t \Phi (C^{-1} t^{-\gamma}) } = O( \rme^{-t^{1-\gamma/\alpha} }), \ \text{ as } t \to 0,
\end{equation*}
As this decays faster than any power of $t$ as $t\to 0$, the upper bound follows.

\subsection{Multiplicative cascade}
Although it is ambiguous what multifractality means, some models are usually studied in this sense. One of the first models of this kind is the multiplicative cascade. Cascades are actually measures, but can be used to construct non-negative increasing multifractal processes. Discrete cascades satisfy only discrete scaling invariance, in the sense that Definition \ref{defM} is valid only for discrete time points. Another drawback of these processes is the non-stationarity of increments.

In \cite{bacry2003}, a class of measures has been constructed having continuous scaling invariance and called multifractal random measures, thus generalizing the earlier cascade models. We will refer to a process obtained from these measures simply as the cascade. Since this is a notable example of a theoretically well developed multifractal process, we analyze it in the view of the results of the preceding section. Furthermore, we consider only one cascade process, the log-normal cascade. One can use cascades as subordinators to BM to build more general models called log-infinitely divisible multifractal processes (see \cite{bacry2003, ludena2008lp} and the references therein).

Following properties hold for the log-normal cascade $\{ X(t)\}$ with parameter $\lambda^2$ (\cite{bacry2008continuous}). $\{ X(t)\}$ satisfies Definition \ref{defD} with the random factor $M(c)=c \rme^{2\Gamma_c}$ where $\Gamma_c$ is Gaussian random variable and can therefore be considered as a true multifractal. Moment scaling holds with
$$\tau(q)=q(1+2 \lambda^2)-2 \lambda^2 q^2.$$
Increments of $\{ X(t) \}$ are heavy-tailed with tail index equal to $2/\lambda^2$ and moments of every negative order are finite provided $\lambda^2<1/2$ (see
 \cite[Proposition 5]{bacry2013lognormal}). Although the asymptotic behaviour of the scaling function defined by \eqref{tauhat} is unknown, there are results for the estimator defined by \eqref{tauhatalternative}. Fixed domain asymptotic properties of this estimator for the multiplicative cascade has been established in \cite{ossiander2000statistical} where it was shown that when $j\to \infty$ estimator \eqref{tauhatalternative} tends a.s. to
\begin{equation}\label{LGNtau}
  \tau_{\infty}(q)=
  \begin{cases}
  h_0^{-} q, & \text{if } q\leq q_0^{-},\\
  q(1+2 \lambda^2)-2 \lambda^2 q^2, & \text{if } q_0^{-} < q < q_0^{+}\\
  h_0^{+} q, & \text{if } q\geq q_0^{+},
  \end{cases}
\end{equation}
where
\begin{align}\label{q0+-}
q_0^{+} &= \inf \{ q \geq 1 : q \tau'(q)-\tau(q) + 1 \leq 0 \}=\frac{1}{\sqrt{2 \lambda^2}},\\
q_0^{-} &= \sup \{ q \leq 0 : q \tau'(q)-\tau(q) + 1 \leq 0 \}=-\frac{1}{\sqrt{2 \lambda^2}}
\end{align}
and $h_0^{+}=\tau'(q_0^{+})$, $h_0^{-}=\tau'(q_0^{-})$. So the estimator \eqref{tauhatalternative} is consistent for a certain range of $q$, while outside this interval the so-called linearization effect happens. Similar results have been established in the mixed asymptotic framework (\cite{bacry2010multifractal}); see also \cite{ludena2014estimating} for a different method. The spectrum of the log-normal cascade is supported on the interval $\left[ 1 + 2 \lambda^2 - 2 \sqrt{2\lambda^2}, 1 + 2 \lambda^2 + 2 \sqrt{2\lambda^2}\right]$, given by
\begin{equation*}
d(h)= \inf_{q \in (-\infty,\infty)} \left( hq - \tau(q) +1\right) = 1- \frac{(h-1 - 2 \lambda^2)^2}{8 \lambda^2},
\end{equation*}
and the multifractal formalism holds (\cite{barral2002multifractal}).

Condition $\tau(q)>1$ of Proposition \ref{prop1} yields $q \in (1,1/(2\lambda^2))$. We then get that
$$H^-=1+2\lambda^2 - 2 \sqrt{2 \lambda^2}.$$
This is exactly the left endpoint of the interval where the spectrum of the cascade is defined, in accordance with Proposition \ref{prop1}. This maximal lower bound is achieved for $q=1/\sqrt{2 \lambda^2}=q^+_0$. If $q^-$ is the point at which maximal lower bound $H^-$ is achieved, then
$$\left( \frac{\tau(q)}{q} - \frac{1}{q} \right)'= \frac{1}{q^2} \left( q \tau'(q) - \tau(q) + 1 \right)$$
must be equal to $0$ at $q^-$. This is exactly defined in \eqref{q0+-}. Although the range of finite moments is not relevant for computing $H^-$ in this case, in general it can depend on $\overline{q}$.

Since all negative order moments are finite we get that
$$\widetilde{H^+}=H^+=1+2\lambda^2 + 2 \sqrt{2 \lambda^2}$$
achieved for $q=-1/\sqrt{2 \lambda^2}$. Thus again the bound from Proposition \ref{prop2} is sharp giving the right endpoint of the interval where the spectrum is defined.

\subsection{Multifractal random walk}
With this example we want to show that we may have $\widetilde{H^+} \neq H^+$ and that the definition of the scaling function needs to be adjusted to avoid infinite moments of negative order.
Multifractal random walk (MRW) driven by the log-normal cascade is a compound process $X(t)=B(\theta(t))$ where $B$ is a BM and $\theta$ is the independent cascade process (see \cite{bacry2003}). Multifractal properties of this process are inherited from those of the underlying cascade. $\{ X(t)\}$ satisfies Definition \ref{defD} with the random factor $M(c)=c^{1/2} \rme^{\Gamma_c}$ where $\Gamma_c$ is Gaussian random variable and the scaling function is given by
$$\tau(q)=q\left( \frac{1}{2} + \lambda^2 \right)-\frac{ \lambda^2}{2} q^2.$$
The range of finite moments is $(-1, 1/\lambda^2)$ as explained in Subsection \ref{subsec3.4}. The spectrum is defined on the interval $\left[ 1/2 +  \lambda^2 - \sqrt{2\lambda^2}, 1/2 +  \lambda^2 + \sqrt{2\lambda^2}\right]$ and given by
\begin{equation*}
d(h)= \inf_{q \in (-\infty,\infty)} \left( hq - \tau(q) +1\right) = 1- \frac{(h-1/2 -  \lambda^2)^2}{2 \lambda^2}.
\end{equation*}
Random factor $M(c)$ is the source of multifractality, has the same scaling function, but all negative order moments are finite. Thus we get
\begin{align*}
H^{-} &= 1/2 +  \lambda^2 - \sqrt{2\lambda^2},\\
\widetilde{H^+} &= \frac{3}{2}+\frac{3 \lambda^2}{2},\\
H^+ &= 1/2 +  \lambda^2 + \sqrt{2\lambda^2}.
\end{align*}
$H^{-}$ and $H^+$ give the sharp bounds, while $\widetilde{H^+}$ is affected by the divergence of the negative order moments. This shows that when the multifractal process has infinite negative order moments, one should specify scaling in terms of the random factor.

\section{Robust version of the partition function}\label{sec5}
In Section \ref{sec3} using Corollary \ref{newComplementKC} we managed to avoid the problematic infinite moments of negative order and prove results like Theorem \ref{thm:boundsHsssi} and Corollary \ref{cor:upperboundMF}. When scaling function \eqref{tauhat} is estimated from the data, spurious concavity may appear for negative values of $q$ due to the effect of diverging negative order moments. We use the idea of Corollary \ref{newComplementKC} to develop a more robust version of the partition function.

Instead of using plain increments in the partition function \eqref{partitionfun}, we can use the maximum of some fixed number $m$ of the same length increments. This will make negative order moments finite for some reasonable range and prevent divergencies. The underlying idea also resembles the wavelet leaders method where leaders are formed as the maximum of the wavelet coefficients over some time scale (see \cite{jaffard2004wavelet}). Since $m$ is fixed, this does not affect the true scaling. Same idea can be used for $q>0$ as Lemma \ref{lemma:kolconditions} indicates this condition can also explain the spectrum. It is important to stress that the estimation of the scaling function makes sense only if the underlying process is known to possess scaling property of the type \eqref{mfdefgeneral}.

Suppose $\{X(t)\}$ has stationary increments and $X(0)=0$. Divide the interval $[0,T]$ into $\lfloor T / (m \Delta t) \rfloor$ blocks each consisting of $m$ increments of length $\Delta t$ and define the modified partition function:
\begin{equation}\label{modifiedpartitionfun}
\widetilde{S}_q(T,\Delta t) = \frac{1}{\lfloor T / (m \Delta t) \rfloor} \sum_{i=1}^{\lfloor T /  (m \Delta t) \rfloor} \max_{l=1,\dots,m} \left| X ( i m \Delta t+l \Delta t) -  X ( i m \Delta t+ (l-1) \Delta t) \right|^q.
\end{equation}
One can see $\widetilde{S}_q(T,\Delta t)$ as a natural estimator of \eqref{thmmomentcond}. Analogously we define the modified scaling function as in \eqref{tauhat} by using $\widetilde{S}_q(n,\Delta t_i)$:
\begin{equation}\label{modifiedtauhat}
\widetilde{\tau}_{N,T}(q) = \frac{\sum_{i=1}^{N}  \ln {\Delta t_i}  \ln \widetilde{S}_q(n,\Delta t_i) - \frac{1}{N} \sum_{i=1}^{N} \ln {\Delta t_i} \sum_{j=1}^{N} \ln \widetilde{S}_q(n,\Delta t_i) }{ \sum_{i=1}^{N} \left(\ln {\Delta t_i}\right)^2 - \frac{1}{N} \left( \sum_{i=1}^{N} \ln {\Delta t_i} \right)^2 }.
\end{equation}
One can alter the definition only for $q<0$ although there is no much difference between two forms when $q>0$.

To illustrate how this modification makes the scaling function more robust we present several examples comparing \eqref{tauhat} and \eqref{modifiedtauhat}. We generate sample paths of several processes and estimate the scaling function by both methods. We also estimate the spectrum numerically using \eqref{formalism}. Results are shown in Figures \ref{Fig1}-\ref{Fig4}. Each figure shows the estimated scaling functions and the estimated spectrum by using standard definition \eqref{tauhat} and by using \eqref{modifiedtauhat}. We also added the plots of the scaling function that would yield the correct spectrum via multifractal formalism and the true spectrum of the process.

For the BM (Figure \ref{Fig1}) and the $\alpha$-stable L\'evy process (Figure \ref{Fig2}) we generated sample paths of length $10000$ and we used $\alpha=1$ for the latter. LFSM (Figure \ref{Fig3}) was generated using $H=0.9$ and $\alpha=1.2$ with path length $15784$ (see \cite{stoev2004simulation} for details on the simulation algorithm used). Finally, MRW of length $10000$ was generated with $\lambda^2=0.025$ (Figure \ref{Fig4}). For each case we take $m=20$ in defining the modified partition function \eqref{modifiedpartitionfun}.

In all the examples considered, the modified scaling function is capable of yielding the correct spectrum of the process with the multifractal formalism. As opposed to the standard definition, it is unaffected by diverging negative order moments. Moreover, it captures the divergence of positive order moments which determines the shape of the spectrum.

\begin{figure}[H]
  \centering
  \subfigure[]{\centering \includegraphics[width=0.43\textwidth]{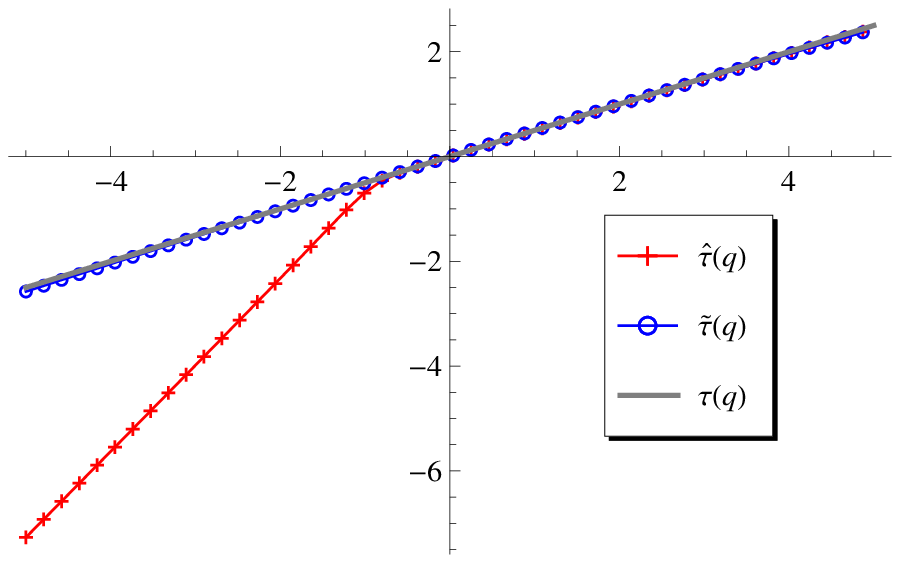} \label{Fig1tau}} \qquad
  \subfigure[]{\centering\includegraphics[width=0.43\textwidth]{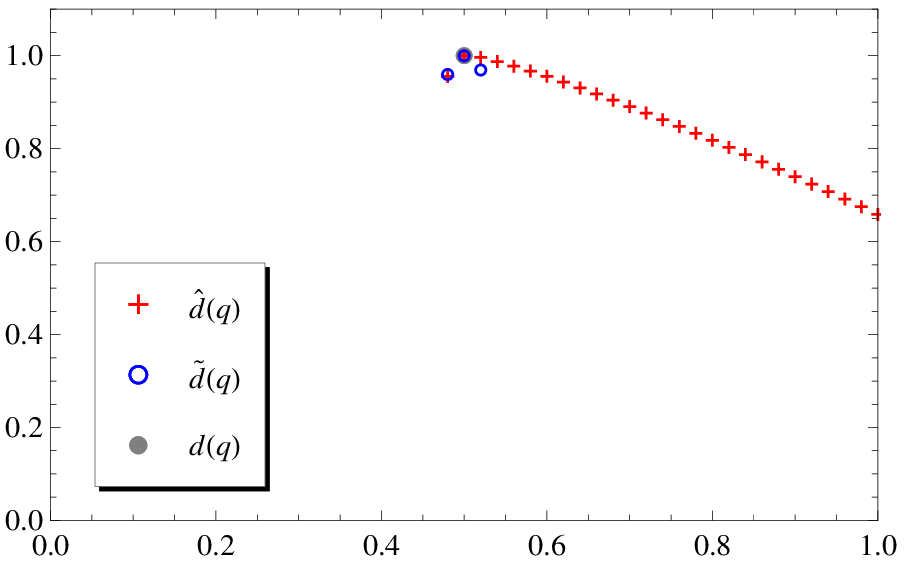} \label{Fig1spe}}
  \caption{Brownian motion - scaling functions (a) and spectrum (b)}  \label{Fig1}
\end{figure}

\begin{figure}[H]
  \centering
  \subfigure[]{\centering \includegraphics[width=0.43\textwidth]{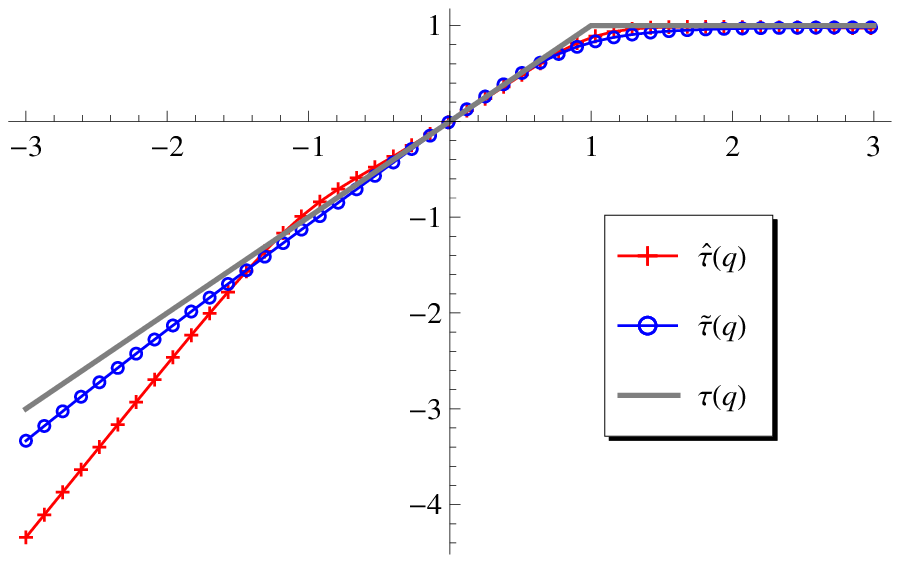} \label{Fig2tau}} \qquad
  \subfigure[]{\centering\includegraphics[width=0.43\textwidth]{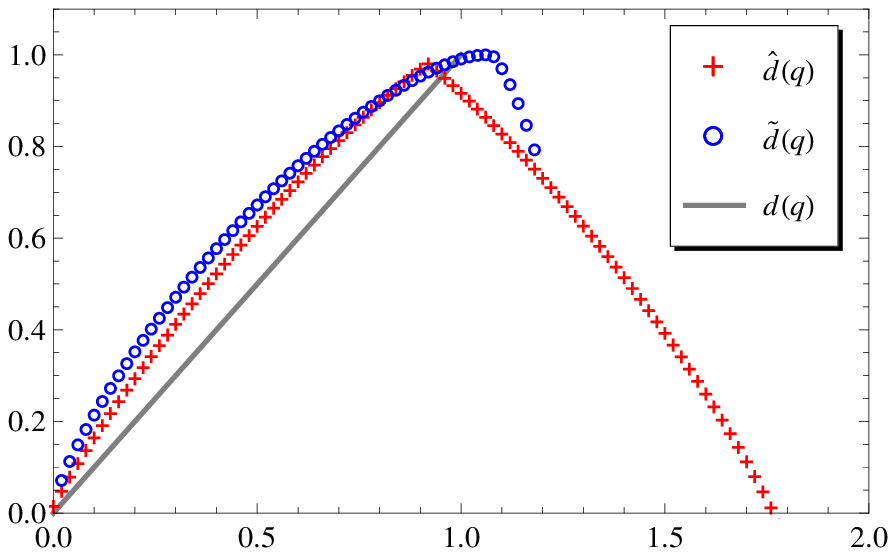} \label{Fig2spe}}
  \caption{Stable L\'evy process $\alpha=1$ - scaling functions (a) and spectrum (b)}  \label{Fig2}
\end{figure}

\begin{figure}[H]
  \centering
  \subfigure[]{\centering \includegraphics[width=0.43\textwidth]{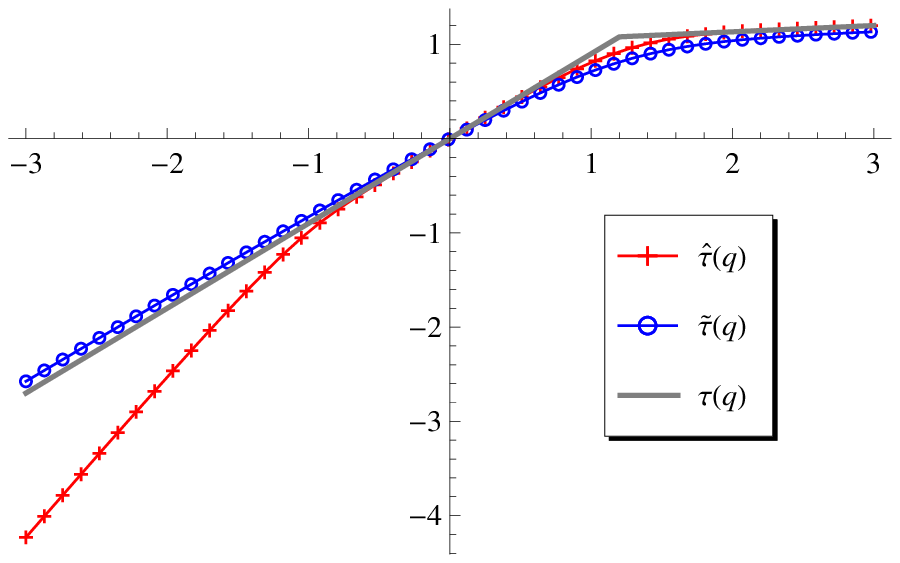} \label{Fig3tau}} \qquad
  \subfigure[]{\centering\includegraphics[width=0.43\textwidth]{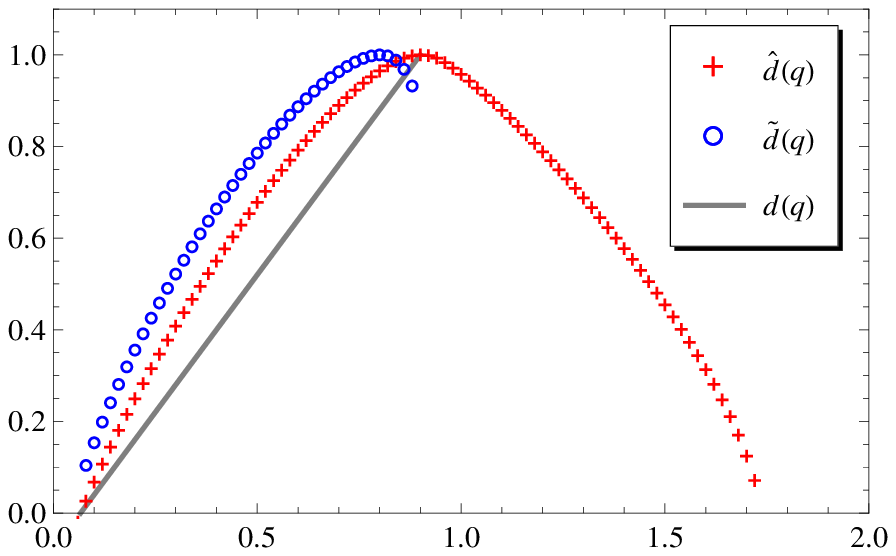} \label{Fig3spe}}
  \caption{Linear fractional stable motion $H=0.9$, $\alpha=1.2$ - scaling functions (a) and spectrum (b)}  \label{Fig3}
\end{figure}

\begin{figure}[H]
  \centering
  \subfigure[]{\centering \includegraphics[width=0.43\textwidth]{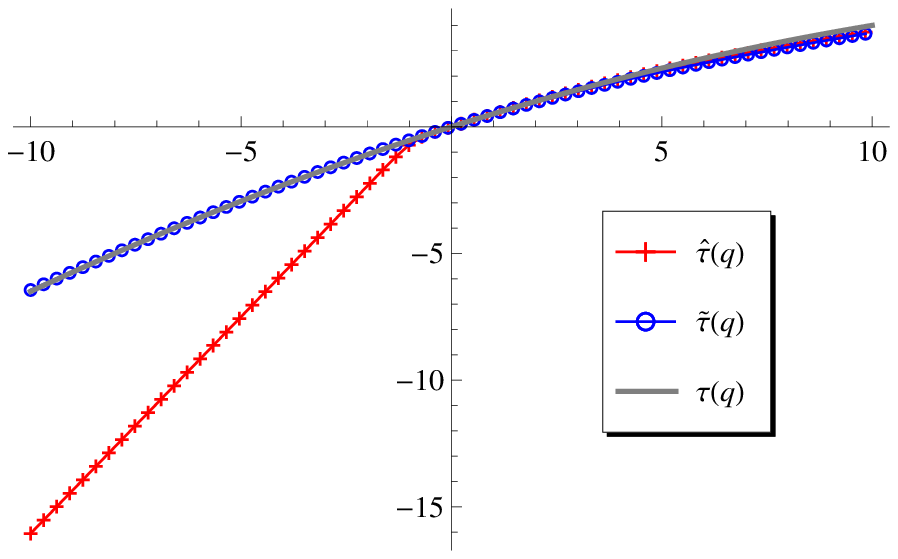} \label{Fig4tau}} \qquad
  \subfigure[]{\centering\includegraphics[width=0.43\textwidth]{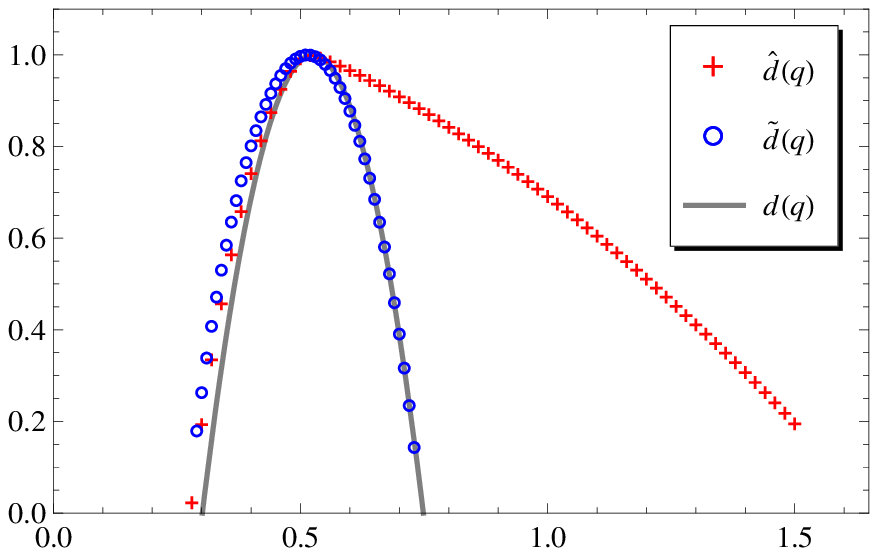} \label{Fig4spe}}
  \caption{Multifractal random walk $\lambda^2=0.025$ - scaling functions (a) and spectrum (b)}  \label{Fig4}
\end{figure}

\section*{Appendix}

We provide a brief overview of different classes of stochastic processes that are used along the paper.

Hermite process $\{Z_H^{(k)}(t), t \geq 0 \}$ with $H \in (1/2,1)$ and $k \in \mathbb{N}$ can be defined as
\begin{equation*}
Z_H^{(k)} (t) = C(H,k) \int_{\mathbb{R}^k}^{'}  \int_0^t \left( \prod_{j=1}^k (s-y_j)_{+}^{-( \frac{1}{2} + \frac{1-H}{k})} \right) \rmd s \rmd B(y_1) \cdots \rmd B(y_k), \quad t \geq 0,
\end{equation*}
where $\{B(t) \}$ is the standard BM and the integral is taken over $\mathbb{R}^k$ except the hyperplanes $y_i=y_j$, $i \neq j$. Constant $C(H,k)$ is chosen such that the $E [Z_H^{(k)}(1) ]^2=1$ and $(x)_{+}=\max (x,0)$. Hermite processes are $H$-sssi. For $k=1$ one gets the FBM and for $k=2$ Rosenblatt process. See e.g. \cite{embrechts2002} for more details.

L\'evy process is a process with stationary and independent increments starting form $0$. Given an infinitely divisible distribution there exists a  L\'evy process such that $X(1)$ has this distribution. Moreover, characteristic function can be uniquely represented by the L\'evy-Khintchine formula. See \cite{bertoin1998levy} and \cite{schoutens2003levy} for more details.

$\alpha$-stable L\'evy process is a process such that $X(1)$ has stable distribution with stability index $0<\alpha<2$. In general, a random variable $X$ has an $\alpha$-stable distribution with index of stability $\alpha \in (0,2)$, scale parameter $\sigma \in (0,\infty)$, skewness parameter $\beta\in [-1,1]$ and shift parameter $\mu \in \mathbb{R}$, denoted by $X \sim S_{\alpha}(\sigma,\beta,\mu)$ if its characteristic function has the following form
\begin{equation*}
E \exp \{ i\zeta X \} =
\begin{cases}
\exp\left\{-\sigma^{\alpha}|\zeta|^{\alpha}\left( 1-i\beta{\rm sign}(\zeta)\tan{\frac{\alpha\pi}{2}}+i\zeta\mu\right) \right\}, & \text{if } \alpha\ne1,\\
\exp\left\{-\sigma|\zeta|\left(1-i\beta\frac{2}{\pi}{\rm sign}(\zeta)\ln{|\zeta|}+i\zeta\mu\right)\right\}, & \text{if } \alpha=1,
\end{cases}
\quad \zeta \in \mathbb{R}.
\end{equation*}
Stable L\'evy process is $1/\alpha$-sssi.

Linear fractional stable motion (LFSM) is an example of a process with heavy-tailed and dependent increments. LFSM can be defined through the stochastic integral
\begin{equation*}
X(t)=\frac{1}{C_{H,\alpha}} \int_{\mathbb{R}} \left( (t-u)_{+}^{H-1/\alpha} - (-u)_{+}^{H-1/\alpha} \right) dL_{\alpha}(u),
\end{equation*}
where $\{L_{\alpha}\}$ is a strictly $\alpha$-stable L\'evy process, $\alpha \in (0,2)$, $0<H<1$ and $(x)_{+}=\max (x,0)$. The constant $C_{H,\alpha}$ is chosen such that the scaling parameter of $X(1)$ equals $1$, i.e.
\begin{equation*}
C_{H,\alpha} = \left( \int_{\mathbb{R}} \left| (1-u)_{+}^{H-1/\alpha} - (-u)_{+}^{H-1/\alpha} \right|^{\alpha} du \right)^{1/ \alpha}.
\end{equation*}
It is then called standard LFSM. The LFSM is $H$-sssi. Setting $\alpha=2$ in the definition reduces the LFSM to the FBM. By analogy to this process, the case $H>1/\alpha$ is referred to as a long-range dependence and the case $H<1/\alpha$ as negative dependence. The parameter $\alpha$ governs the tail behaviour of the marginal distributions implying, in particular, that $E|X(t)|^q=\infty$ for $q \geq \alpha$. For more details see \cite{samorodnitsky1994}.

A L\'evy process $\{Y(t)\}$ such that $Y(1) \sim S_{\alpha}(\sigma,1,0)$, $0<\alpha<1$ is referred to as the stable subordinator. It is nondecreasing and $1/\alpha$-sssi. Inverse stable subordinator $\{ X(t)\}$ is defined as
\begin{equation*}
X(t) = \inf \left\{ s >0 : Y(u) >t \right\}.
\end{equation*}
It is $\alpha$-ss with dependent, non-stationary increments and corresponds to a first passage time of the stable subordinator strictly above level $t$. For more details see \cite{meerschaert2013inverse} and references therein.

\bigskip



\bibliographystyle{agsm}
\bibliography{References}

\end{document}